\documentclass[11pt]{amsart}

\usepackage{amssymb}
\usepackage[all]{xy}
\usepackage[colorlinks=true, urlcolor=rltblue, citecolor=drkgreen, linkcolor=drkred] {hyperref}
\usepackage{color}
\usepackage{pdfsync}
\definecolor{rltblue}{rgb}{0,0,0.4}
\definecolor{drkred}{rgb}{0.6,0,0}
\definecolor{drkgreen}{rgb}{0,0.4,0}
\usepackage{graphicx}
\usepackage[mathscr]{eucal}

\usepackage{lmodern}
\usepackage[capitalize]{cleveref}
\usepackage{amssymb,amsmath,amsthm}
\usepackage{mathtools, MnSymbol}
\usepackage{setspace}
\usepackage{tikz-cd}
\usetikzlibrary{decorations.pathreplacing}
\usepackage[T1]{fontenc}
\usepackage[utf8]{inputenc}
\usepackage{textcomp} % provide euro and other symbols
\usepackage{stmaryrd}
\usepackage{datetime}
\usepackage{todonotes}
\usepackage{xcolor}
\usepackage{mdframed}
\usepackage[all]{xy}
\usepackage{thmtools}
\usepackage{enumitem}

\declaretheorem[numberwithin=section]{theorem}
\declaretheorem[sibling=theorem]{lemma}
\declaretheorem[sibling=theorem]{proposition}
\declaretheorem[sibling=theorem]{corollary}

\declaretheorem[sibling=theorem]{convention}
\declaretheorem[numberwithin=theorem]{claim}

\theoremstyle{definition}
\declaretheorem[sibling=theorem]{definition}

\newcommand{\dSinf}[1]{d\text{-}\Sigma^{\mathrm{in}}_{#1}}

\newcommand{\bigwwedge}{%
  \mathop{
    \mathchoice{\bigwedge\mkern-15mu\bigwedge}
               {\bigwedge\mkern-12.5mu\bigwedge}
               {\bigwedge\mkern-12.5mu\bigwedge}
               {\bigwedge\mkern-11mu\bigwedge}
    }
}

\def\sA{\mathcal A}

\def\A{\mathcal A}
\def\B{\mathcal{B}}
\def\C{\mathcal{C}}
\def\D{{\mathcal D}}
\def\bN{{\mathbb N}}

\def\bQ{{\mathbb Q}}
\def\E{\mathcal{E}}

\def\K{\mathcal K}
\def\L{\mathcal L}

\def\sL{\mathcal L}
\def\sM{{\mathcal M}}

\def\P{{\mathcal P}}

\def\S{{\mathcal S}}

\def\bQ{{\mathbb Q}}

\def\la{\langle}
\def\ra{\rangle}

\title[Homogeneous Linear Orderings]{$sp$-Homogeneous Linear Orderings}

\author[Calvert]{Wesley Calvert}
\address{School of Mathematical and Statistical Sciences\\ Mail Code 4408\\
Southern Illinois University, Carbondale\\
1245 Lincoln Drive\\
Carbondale, Illinois 62901}
\email{wcalvert@siu.edu}
\author[Cenzer]{Douglas Cenzer}
\address{Department of Mathematics, University of Florida,
Gainesville, FL 32611}
\email{cenzer@ufl.edu}
\author[Gonzalez]{David Gonzalez}
\address{Department of Mathematics, University of Notre Dame, Notre Dame, IN 46556}
\email{dgonza42@nd.edu}
\author[Harizanov]{Valentina Harizanov}
\address{Department of Mathematics, George Washington
University, Washington, DC 20052}
\email{harizanv@gwu.edu}
\author[Ng]{Keng Meng Ng}

\thanks{The authors' research at the Hausdorff Institute for Mathematics, Bonn in Fall 2025, was partially funded by the Deutsche Forschungsgemeinschaft (DFG, German Research Foundation) under Germany's Excellence Strategy – EXC-2047/1–390685813.  The authors were also partially supported by NSF grant DMS-2401437. Harizanov was also partially supported by FRG NSF grant DMS-2152095.}

\date{\today}

\begin{document}

\begin{abstract} We study linear orderings expanded by functions for successor and predecessor. The successor and predecessor on linear orderings capture the relatively intrinsically computably enumerable information about orderings in much the same way that dependence captures that for vector spaces. In particular, the \textit{sp}-homogeneous and weakly \textit{sp}-homogeneous linear orderings are those which are (ultra-)homogeneous or weakly homogeneous with this additional structure.
 We demonstrate that these orderings are always relatively $\Delta_4$ categorical and determine exactly which ones are (uniformly) relatively $\Delta_3$ categorical.

    We also provide a classification for \textit{sp}-homogeneity and weak \textit{sp}-homogeneity.
    We establish that this is the best possible classification by showing that the set of \textit{sp}-homogeneous linear orderings is $\Pi_5^0$ complete, and that the set of weakly \textit{sp}-homogeneous linear orderings is $\Sigma_6^0$ complete.
    These results are obtained in two different ways, one using a hands-on computability theoretic approach and another using more abstract descriptive set theory.
\end{abstract}

\maketitle

\section{Introduction}

We aim to extend the class of well-understood, simple linear orderings and calibrate the exact levels of complexity in this class by blending various methods from logic.
One of the most well-known and well-studied simplicity notions comes from model theory.
A structure $\A$ is said to be \emph{homogeneous} if any isomorphism
between finitely generated substructures extends to an automorphism
of $\A$.
This notion is sometimes also called ultra-homogeneous, but we will use the term homogeneous throughout the paper.

Homogeneous structures were first studied
by Fra\"{i}ss\'{e} \cite{Fr86} from the perspective of model theory.
In the many years since then, the study of homogeneous structures has developed into an important endeavor in many fields, including combinatorics, descriptive set theory, permutation group theory, topological dynamics, and representation theory (see Lachlan's ICM survey \cite{L87} or the more complete account by Macpherson \cite{Mac11}).  %In the present paper, we address the role of homogeneity in computable structure theory (see, for example, \cite{M02,CHMM11}).

In \cite{AC17}, Adams and Cenzer defined the notion of \emph{weakly homogeneous} structures,
where $\A$ is weakly homogeneous if there is a
finite (exceptional) set of elements $a_1,\dots,a_n$ of $\A$ such that any isomorphism between finitely generated substructures $\A$ which maps each $a_i$ to itself may be extended to an automorphism of the structure. 
This extension appreciably expands the class of homogeneous structure while still maintaining many of the properties enjoyed by homogeneous structures.

Considerable progress has been made concerning the analysis of weakly homogeneous structures.
The initial analysis in \cite{AC17} focused on linear orderings, equivalence structures, injection structures, and trees.
In each case, the weakly homogeneous structures were characterized.
For instance, it is well known that the only countable, homogeneous linear ordering is $(\mathbb{Q},<)$, the order type of the rationals. It was shown that the only weakly homogeneous linear orderings are those with finitely many successivities.
These are also known to be the linear orderings that are \textit{relatively computably categorical}, i.e., such that any two copies $\A$ and $\B$ have an isomorphism $f:\A\to\B$ that can be given by an algorithm based only on the information of $\A$ and $\B$.
This is not a coincidence.
The key result of \cite{AC17} shows that there is an interesting connection between weak homogeneity and computability-theoretic properties of a structure.
In particular, a structure $\A$ is weakly homogeneous, then it is relatively $\Delta^0_2$ categorical.
Moreover, $\A$ is also locally finite; it is relatively computably categorical.
The notion of relative computable categoricity is well studied in computable structure theory as an alternate notion of simplicity for structures.
Extensions of this idea, such as the relative $\Delta_2^0$ categoricity mentioned above, form a hierarchy of complexity above relative computable categoricity, with each level allowing more computational power to calculate isomorphisms between copies of the given structure.

This notion of complexity is of independent interest.
For example, Dzoev and Goncharov \cite{DG1980}  and independently Remmel \cite{Rem81} characterized the relatively computably categorical linear orderings. Furthermore, McCoy \cite{McCoy2003} characterized the relatively $\Delta_2$ categorical linear orderings.
The above theorem connects the purely model-theoretic notion of weak homogeneity and the computability-theoretic categoricity hierarchy, suggesting that this fortuitous connection merits further investigation.
An important problem regarding this connection is whether the converse of this result holds in general or for certain classes of structures.
As we observed for plain linear orderings, the converse of the above result holds.
Between the papers \cite{AC17} and \cite{ACN20}, results along these lines were proven for many structures of importance in classical mathematics.
For example, equivalence relations, trees under the predecessor function, injection structures, Boolean algebras, and Abelian $p$-groups have all been analyzed.

We take a different approach to weak homogeneity in this paper.
We consider linear orderings, but not the plain linear orderings considered in \cite{AC17}.
We consider linear orderings equipped with a total successor and predecessor functions, $s$ and $p$. Here $s(x)$ is the successor of $x$, if one exists, and otherwise $s(x) = x$; similarly, $p(x)$ is either the predecessor of $x$ or equals $x$.  The successor and predecessor relations on linear orderings capture the relatively intrinsically computably enumerable information about orderings in much the same way that dependence captures that for vector spaces. 

We study linear orderings that become homogeneous or weakly homogeneous when enriched with this additional structure, also called \textit{sp-homogeneous} or \textit{weakly $sp-$homogeneous} orderings. It is considerably more interesting to study \textit{sp}-homogeneity for linear orderings over plain homogeneity.
It is a far richer class of orderings.
Most saliently, \textit{sp}-homogeneous linear orderings pierce farther into the categoricity hierarchy than their plainly homogeneous cousins.
We see in this paper that all weakly \textit{sp}-homogeneous orderings are relatively $\Delta_4$ categorical, and many of them achieve this level of categoricity.
As mentioned above, the $\Delta_1$ and $\Delta_2$ levels of relative categoricity are well-understood for linear orderings (\cite{DG1980,Remmel1981, McCoy2002}), and, indeed, all orderings at these levels are $sp$-homogeneous.
In this, $sp$-homogeneity captures the entire class of simple orderings that are already well-understood.

In contrast, there is good reason to believe that even relative $\Delta_3$ categoricity for linear orderings is fundamentally difficult to understand in full generality.
This was first argued in \cite{McCoy2002}, and \cite{GR23} puts a point on this by showing a uniform isomorphism preserving way to transform any relatively computable categorical structure into a relatively $\Delta_3$ categorical linear ordering.
This means that analyzing specific subclasses of orderings at the $\Delta_3$ level is required to make significant progress at this point of the categoricity hierarchy.
Linear orderings that are $sp$-homogeneous or weakly $sp$-homogeneous provide a large class of tangible examples that allow the exploration of these difficult-to-understand levels of complexity.
This concreteness yields results; we provide a full classification of the categoricity levels of the $sp$-homogeneous linear orderings.
Since the $\Delta_1$ and $\Delta_2$ levels are well-understood, this comes down to distinguishing the $\Delta_3$ case from the $\Delta_4$ case.
An interesting feature of our arguments that contrasts previous work is that we find a complexity division with continuum many models on each side of the line.
Previous work depended on listing the simpler models explicitly.
We achieve this by enumerating a set of structural ``forbidden arrangements'' that, when avoided, ensure the relative $\Delta_3$ categoricity of an $sp$-homogeneous linear ordering.
 \begin{theorem}
        An $sp$-homogeneous linear ordering $L$ is relatively $\Delta_3$ categorical if and only if the following holds. 
        \begin{enumerate}
            \item It has no intervals of the form $Sh(S)$ where $S$ includes an infinite block and finite blocks of arbitrary size or $\zeta$ along with another finite block.
            \item If $I$ is $\omega\cdot\eta$, $Sh(\omega,\omega^*)$, $\omega^*\cdot\eta$, or $\zeta\cdot\eta$ or a sum of at least two of those orderings, it has an interval to its left and right in $L$ that only has finite blocks of a bounded size.
            \item Any $\zeta\cdot\eta$ does not have an interval to its right or left isomorphic to a shuffle sum including an infinite block.
        \end{enumerate}
\end{theorem}
This theorem provides the most robust understanding of the $\Delta_3$ level of the categoricity hierarchy for linear orderings currently known.
For example, it subsumes the result in \cite{McCoy2002} that demonstrates that all shuffle sums of finite blocks are relatively $\Delta_3$ categorical.

En route to proving this theorem, we establish multiple tools needed to understand the concepts of $sp$-homogeneity and weak $sp$-homogeneity.
This means that, along the way to our main analysis, we establish independently interesting computability-theoretic facts about $sp$-homogeneity.
Critical to our analysis, for example, is a structural characterization of the $sp$-homogeneous and weakly $sp$-homogeneous orderings that comes in the following theorem.
All relevant terms used in the theorem are provided in the paper; the critical definition is that of the \textit{shuffle sum} of a countable set of orderings, $Sh(A)$, which densely arranges each of the orderings in $A$.
\begin{theorem} Let $L$ be a linear ordering. Then
$L$ is $sp$-homogeneous if and only if there is a pairwise disjoint family $\{A,A_v, v \in V\}$ of subsets of $\omega \cup \{\omega,\omega^*,\zeta\}$ such that $L$ is the union of suborderings, as follows: 

\begin{enumerate}
\item For each $v \in V$, an open interval $I_v$ isomorphic to $Sh(A_v)$;

\item For each $a \in A$, a single block of size $a$.
\end{enumerate}
\bigskip
$L$ is \emph{weakly} $sp$-homogeneous if and only if it can be written as
\[L=L_1+B_1+\cdots+B_{k-1}+L_{k},\]
where each $L_i$ is (possibly empty) $sp$-homogeneous and each $B_i$ is a single (non-empty) block.
\end{theorem} 
This characterization is entirely structural and concrete, but it is rather complicated. We note that this  complexity is necessary in the definition, as evident from our next two closely related theorems, one from a computability-theoretic perspective and the other from a more descriptive set-theoretic perspective.
\begin{theorem}
\begin{enumerate}
    \item $\{e: L_e \ \text{is $sp$-homogeneous}\}$ is $\Pi^0_5$ complete.
    \item $\{e: L_e \ \text{is weakly $sp$-homogeneous}\}$ is $\Sigma^0_6$ complete.  
\end{enumerate}
\end{theorem}

\begin{theorem}
\begin{enumerate}
    \item $\{\A\in Mod(LO): \A\ \text{is $sp$-homogeneous}\}$ is $\mathbf{\Pi}^0_5$ complete.
    \item $\{\A\in Mod(LO): \A \ \text{is weakly $sp$-homogeneous}\}$ is $\mathbf{\Sigma}^0_6$ complete.  
\end{enumerate}
\end{theorem}
\noindent Both versions of the complexity levels are explicitly laid out here because we offer two distinct perspectives on proving these sorts of results in our paper.
Of particular import in the boldface setting is the use of the back-and-forth technique.
Understanding the also-called back-and-forth relations is a way to understand the limitations of the Borel hierarchy (for more details see Chapter II in \cite{cst2}). This technique is also used extensively in the main categoricity classification result.

We should also note current paper \cite{GComb}, which analyzes the combinatorial content of $sp$-homogeneous orderings.
In particular, it is shown that $sp$-homogeneity is equivalent to a notion called $C_{\infty,\infty}$-homogeneity, which is homogeneity relative to an expansion of linear ordering by infinitely many additional relations.
By restricting which of these relations are allowed, strong homogeneity notions, also called $C_{n,m}$-homogeneity, are investigated and the number of $C_{n,m}$-homogeneous orderings is determined.
A full, detailed account can be found in \cite{GComb}.
The connections with finite enumerative combinatorics give another reason to be interested in $sp$-homogeneous orderings.

The paper is structured in five sections, including this one containing introductory material.
In Section 2, we structurally classify $sp$-homogeneous and weakly $sp$-homogeneous linear orderings.
In Section 3, we show that this classification is best possible from a computability-theoretic perspective using index sets.
Section 4 mirrors the third section and shows that our classification is the best possible from a boldface, descriptive set-theoretic perspective.
Lastly, in Section 5, we provide the analysis of categoricity level for $sp$-homogeneous orderings.

\section{Classifying Weakly $sp$-Homogeneous Linear Orderings}
We prove basic properties about $sp$-homogeneity and give a structural characterization of the $sp$-homogeneous and weakly $sp$-homogeneous linear orderings.

We begin with our characterization of the weakly \textit{sp}-homogeneous linear orderings $L$, starting with the \textit{sp}-homogeneous ones. Recall that $Bk(L)$ is the linear ordering of the blocks of $L$.
More formally, $Bk(L)$ is the quotient of $L$ by the convex equivalence relation identifying elements that are finitely far from each other.
We give the following preliminary definition regarding a fundamental construction of linear orderings.

\begin{definition}
    Given a countable set of linear orderings $A=\{A_1,A_2,\dots\}$ the shuffle sum of $A$, denoted $Sh(A)$, is obtained by partitioning $\eta$ into $|A|$ dense
sets $K_i$ and replacing each point in $K_i$ with a copy of $A_i$.
\end{definition}

It is a standard fact that the shuffle sum of a countable set of linear orderings is well-defined up to isomorphism.
This is confirmed by a back-and-forth argument extending the uniqueness of the dense linear ordering without endpoints.
We use several important properties regarding the sums of shuffle sums throughout the paper.
Let $\K$ be an arbitrary linear ordering and $P\in A$.
We have that
\[Sh(A)\cdot\K\cong(Sh(A)+P)\cdot K +Sh(A)\cong Sh(A).\]
These equations follow from the uniqueness of the shuffle sum.
The special case that $Sh(A)+Sh(A)\cong Sh(A)+P+Sh(A)\cong Sh(A)$ is particularly useful.

\begin{theorem} \label{ush} For any $A \subseteq \omega\cup\{\omega,\omega^*,\zeta\}$, the shuffle sum $L = (Sh(A),s,p)$ is $sp$-homogeneous. 
\end{theorem}

\begin{proof} Given finite sequences $a_0 < a_1 < \dots < a_k$ and $b_0 < b_1 < \dots < b_k$ of elements from $L$, the $s$ and $p$ functions will generate the blocks $[a_i]$ and $[b_i]$. Assume, without loss of generality, that these blocks are distinct. If these sets of blocks are isomorphic, then this is easily extended to an automorphism of $L$, since the intervals between any two distinct blocks are all isomorphic to $Sh(A)$. 
\end{proof}

We now show that our above example of a shuffle sum of blocks being $sp$-homogeneous is typical.
More specifically, the shuffle sums of blocks are the components of more complicated $sp$-homogeneous orderings.

\begin{theorem}\label{thm:homogChar}  Let $L$ be a linear ordering.
$L$ is $sp$-homogeneous if and only if there is a pairwise disjoint family $\{A,A_v, v \in V\}$ of subsets of $\omega \cup \{\omega,\omega^*,\zeta\}$ such that $L$ is the union of suborderings, as follows: 

\begin{enumerate}
\item For each $v \in V$, an open interval $I_v$ isomorphic to $Sh(A_v)$;

\item For each $a \in A$, a single block of size $a$.
\end{enumerate}
\bigskip
$L$ is \emph{weakly} $sp$-homogeneous if and only if it can be written as
\[L=L_1+B_1+\cdots+B_{k-1}+L_{k},\]
where each each $L_i$ is (possibly empty) $sp$-homogeneous and each $B_i$ is a single (non-empty) block.
\end{theorem} 

\begin{proof} 
We begin with the proof for $sp$-homogeneity.

First, suppose that $L$ has the described form and let $\phi: X \to Y$ be an isomorphism between two finitely generated subsets $X$ and $Y$ of $L$. 
For any $x$ with block type $b(x) \in A$, we must have $\phi(x) = x$. For any $v \in V$ and any $x \in I_v$, we must have $\phi(x) \in A_v$ and $\phi(x)$ must have the same block type as $x$, so that $\phi(x) \in I_v$. Thus $\phi$ acts as an isomorphism from $X \cap I_v$ to $X \cap I_v$.  Since $Sh(A_v)$ is 
$sp$-homogeneous by Theorem \ref{ush}, this may be extended to an automorphism 
$\alpha_v$ of $I_v$. Then the desired automorphism of $L$ is defined by letting $\alpha(x) = \alpha_v(x)$ when $x \in I_v$ where $X \cap I_v \neq \emptyset$, and $\alpha(x) = x$ for all other elements. 

\bigskip

Next, suppose that $L$ is $sp-$homogeneous.
Say that $Bk(L)$ is isomorphic to $W$, so $L$ has the form $\bigcup_{q \in W} b(q)$, where each $b(q)$ is a block of type $t(q)$.
The set $A$ is defined to be the set of block types occurring only once in $L$. 
Let $v = t(q)$ be one of the types occurring in $Bk(L)$. 

\begin{claim}
If $v$  appears more than once, then there is an open interval $I_v\subseteq W$ containing all $q$ such that $t(q) = v$. Moreover, the set of $q$ in $I_v$ with $t(q) = v$ is dense in $I_v$ and has no first or last element. 
\end{claim} 

\begin{proof}
First, we show that there is no least $x$ with $t(x) = v$.
To see this, suppose by way of contradiction that there is a least $x$ with $t(x) = v$ and let $b(y)$ be any other block in $I_v$ with $t(y) = v$. Then the map taking the block $b(x)$ to $b(y)$ cannot be extended to an automorphism of $L$, since $y$ has a block of size $v$ below it, and $x$ does not. A similar argument shows that there is no maximal $x$ with $t(x) = v$. Now suppose that $x_1 < x_2$ and $t(x_1) = v = t(x_2)$ but there is no $y$ in $(x_1,x_2)$ with $t(y) = v$. By the above argument, there is some $x_0 < x_1$ with $t(x_0) = v$. Then the map taking $b(x_0)$ to $b(x_1)$ and $b(x_2)$
to $b(x_2)$ cannot be extended to an automorphism, since there is a block of type $v$ between $b(x_0)$ and $b(x_2)$, but there is no such block between $b(x_1)$ and $b(x_2)$. 

The interval $I_v$ may be defined as the set of $y$ such that there are $x,z$ with block size $v$ such that $x < y < z$.
\end{proof}

Next, we show that the interval $I_v$ will, in fact, be a shuffle sum. The key here is that for any other block type $u$ appearing in $I_v$, we have that $I_u$ must equal $I_v$. 
Let $A_v$ be the set of types $t(q)$ such that $q \in I_v$. 
\begin{claim}\label{claim:spHomShuffle}
    For any type $u \in A_v$, the set of blocks of type $u$ in $I_v$ is dense in $I_v$ and has no first or last element.
    In particular, $I_v=Sh(A_v)$.
\end{claim}

\begin{proof}
     First, suppose that $y$ were the least element of $I_v$ of type $u$. Choose $x_0$ and $x_1$ of type $v$ from different blocks such that $x_0 < y < x_1$. Then the map taking $b(x_0)$ to $b(x_1)$ cannot be extended to an automorphism, since $x_1$ has a block of size $u$ below it, but $x_0$ does not. A similar argument shows that there is no greatest element of type $u$ in $I_v$.
     Suppose that $y_0 < y_1$ from distinct blocks in $I_v$ have block type $u$, but there is no $y \in (y_0,y_1)$ of block type $u$. Choose $x_0 < y_0$ and $x_1 \in (y_0,y_1)$ of block type $v$. Then the map taking $x_0$ to $x_1$ and $y_1$ to $y_1$ cannot be extended to an automorphism, since there is an element of block type $u$ in $(x_0,y_1)$ but there is no such element in $(x_1,y_1)$. Finally, we need to check that for any $x_0 < x_1$ in $I_v$ and any type $u \in A_v$, there is an element of type $u$ in $(x_0,x_1)$.  Suppose that $x_i$ has block type $u_i$ for $u=0,1$. By the previous argument, there is an element $y_1$ of block type $u$ greater than $x_1$. Then there is also an element $z$ of block type $u_1$ greater than $y_1$. Consider the isomorphism mapping $x_0$ to $x_0$ and $x_1$ to $z$. If there were no element of block type $u$ in $(x_0,x_1)$, then this could not be extended to an automorphism, since there is an element of type $u$ in $(x_0,z)$. This completes the proof that $I_v$ is a shuffle sum of $A_v$.
\end{proof}

Claim \ref{claim:spHomShuffle} completes the proof of the theorem for the case concerning $sp$-homogeneous orderings.
We now need only prove the associated claim for weak $sp$-homogeneity. 
Let $L$ be weakly $sp$-homogeneous over the exceptional set $\bar{p}=p_1<\cdots<p_r$.
By convention, let $p_0=-\infty$ and $p_{r+1}=\infty$.
Let $(p_i,p_{i+1})_B$ denote the set of elements whose block is strictly between the block of $p_i$ and that of $p_{i+1}$
For the sake of contradiction, say that some interval $(p_i,p_{i+1})_B$ is not $sp$-homogeneous.
This is witnessed by some offending pair $\bar{a},\bar{b}$ which generate the same structure yet are not automorphic.
That said, any automorphism of $L$ fixing $\bar{p}$ must restrict to an automorphism of each $(p_i,p_{i+1})_B$.
Therefore, $\bar{a},\bar{b}$ are not automorphic over $\bar{p}$ in $L$, yet they still generate the same structure over $\bar{p}$, a contradiction to the fact that $L$ is weakly $sp$-homogeneous over the exceptional set $\bar{p}=p_1<\cdots<p_r$.
This means that each $(p_i,p_{i+1})_B$ is $sp$-homogeneous, yielding the desired structural form for $L$.

On the other hand, say that $L=L_1+B_1+\cdots+B_{k-1}+L_{k}$ where each $L_i$ is $sp$-homogeneous and each $B_i$ is a single block.
Take a single element from each $B_i$ as the exceptional set $\bar{p}$.
Consider two tuples $\bar{a}$ and $\bar{b}$ that generate the same structure over $\bar{p}$.
Write $\bar{a}=\bar{a}_1,\dots,\bar{a}_k,\bar{a}'_1,\dots,\bar{a}'_{k-1}$ where each $\bar{a}_i\in L_i$ and $\bar{a}'_i\in B_i$.
Write $\bar{b}=\bar{b}_1,\dots,\bar{b}_k,\bar{b}'_1,\dots,\bar{b}'_{k-1}$ similarly.
By assumption, for each $i$, $\langle\bar{a}_i\rangle\cong\langle\bar{b}_i\rangle$, and as there is a fixed element in each $B_i$, $\bar{a}'_i=\bar{b}'_i$.
By $sp$-homogeneity, there is an automorphism $\sigma_i$ for $L_i$ taking $\bar{a}_i$ to $\bar{b}_i$.
Taken together, these yield an automorphism of $L,\bar{p}$ that takes $\bar{a}$ to $\bar{b}$.
In particular, $L$ is weakly homogeneous over the exceptional set $\bar{p}$ as desired.
\end{proof}

Here is a more explicit way to visualize these $sp-$homogeneous linear orders. 

\begin{proposition} \label{prop:spVisual} Let $\{A,A_v, v \in V\}$ be a pairwise disjoint family  of subsets of $\omega \cup \{\omega,\omega^*,\zeta\}$ and let $\sM$ be an arbitrary countable linear ordering.  Then an $sp-$homogeneous linear ordering $\sL$ as described in Theorem \ref{thm:homogChar} may be obtained by replacing each element of $\sM$ with either an open interval of type $I_v$, for $v \in V$, or a single block of size $a$, for $a \in A$, in such a way that
\begin{enumerate}
\item Every interval $I_v$ with $v\in V$, and every block of size $a$ with $a \in A$, occur exactly once;
\item No finite block is immediately followed by another finite block or a block of type $\omega$;  
\item A block of type $\omega^*$ does not immediately precede a finite block or a block of type $\omega$.
\end{enumerate}
\end{proposition}

\begin{proof}  It is immediate that $\sL$ as described above satisfies the condition given in Theorem \ref{thm:homogChar}.
\end{proof}

In particular, if $\sM$ is $\eta$,  then the components of $\sL$ could be placed arbitrarily. $\sM = \omega$ could be used as long as the $V$ is infinite, so that we never have to put two of the single blocks next to each other.  

On the other hand, given any $sp$-homogeneous linear ordering with corresponding $\{A,A_v, v \in V\}$, an ordering $\sM$ as in Proposition \ref{prop:spVisual} can be defined as follows. The elements of $\sM$ will consist of a representative $c_v$ from each interval $I_v$ and a representative $c_a$ from the block of type $a$ for each $a \in A$. The ordering on these representatives is inherited from $\sL$.

We note there is a similarity of this result with the classification of homogeneous colored linear orderings found in \cite{ST08} Lemma 3.1.
This similarity is not a coincidence, and it is explored thoroughly in \cite{GComb}.

In the following two results, we refer to the operation of a ``separated sum'' of linear orderings.  Given linear orderings $L,K$, the separated sum $L \oplus K$ is $L + 1 + K$.

\begin{corollary}  Every relatively $\Delta^0_2$ categorical linear ordering is weakly homogeneous as an $sp$-linear ordering. 
\end{corollary}

\begin{proof} Each relatively $\Delta^0_2$ categorical ordering is a finite, separated sum of orderings of the following types: $n$ for finite $n$, $\omega$, $\omega^*$, $\omega+\omega^*$, and $n\cdot \eta$ by the main result in \cite{McCoy2002}. The first three components are just single blocks, and $n \cdot \eta = Sh(\{n\})$.
\end{proof}

Note here that if we only had the successor function, and not the predecessor function, this result would not hold for $\omega$, since then we would have isomorphisms from $\omega$ to the substructure $[n,\infty)$ which could not be extended to an automorphism. The result will still hold for $\zeta$. 

For scattered linear orderings, those that do not accept an embedding from $\eta$, the converse of the above corollary holds.
This is because any scattered weakly $sp-$homogeneous ordering is necessarily a finite sum of blocks by the above characterization.
In other words, for scattered orderings, being relatively $\Delta_2$ categorical is just the same as being weakly $sp$-homogeneous.
We will see many examples of this correspondence failing in the non-scattered case throughout this article.

The following corollary limiting the scope of weak $sp$-homogeneity can be seen as a result of the structure theorem proven above.
We provide a direct proof of the fact as well for maximal clarity.

\begin{corollary}\label{cor:BkHasFinSuccessors}
    Let $L$ be a linear ordering such that $Bk(L)$ has infinitely many successors. Then $L$ is not weakly $sp$-homogeneous.
\end{corollary}

\begin{proof} Let $(a_0,b_0), (a_1,b_1), \dots$ be an infinite set of successor blocks in $L$. For any $i$,  either $a_i$ is infinite or $b_i$ is infinite.  Assume, without loss of generality, that infinitely many $a_i$ are infinite.  Since there are only three infinite types: $\omega$, $\omega^*$, and $\zeta$, there must be infinitely many of the same type, say $\omega$.  Given a fixed finite set of parameters, choose successor blocks $(c_0,d_0) < (c_1,d_1)$ greater than any of the parameters with each $c_i$ of type $\omega$. Consider the isomorphism mapping $c_0$ to $c_1$ and $d_1$ to $d_1$. 
This cannot be extended to an automorphism since there is a block between $c_0$ and $d_1$, but there is no infinite block between $c_1$ and $d_1$. 
\end{proof}

This provides ready examples of well-behaved linear orderings that are not weakly $sp$-homogeneous.
In particular, the following are a well-studied class of orderings.

\begin{definition}
    Given $f:\omega\to\omega$ the $\zeta$ representation $Z_f$ is given by
    \[Z_f=\sum_\omega\zeta+f(i).\]
\end{definition}

It can be directly observed that $Bk(Z_f)\cong\omega$ no matter the choice of $f$.
The following is immediate.

\begin{corollary} \label{wspcor} 
No $\zeta$ representation is weakly $sp$-homogeneous.
\end{corollary}

The notion of Hausdorff rank is important in the study of linear orderings. 
We define the quotient ordering $Bk(\sA)$ of a linear ordering $\sA$  to be the ordering on the set of blocks of $\sA$ with the inherited ordering $[a] \leq [b] \iff a \leq b$.
The block quotient may be iterated by letting $Bk^0(\sA) = \sA$
and $Bk^{n+1}(\sA) = Bk(Bk^n(\sA))$.
This quotient can also be iterated transfinitely with a bit more care.

\smallskip

The \emph{Hausdorff rank} of $\sA$ is the least $\alpha$ such that $Bk^{\alpha+1}(\sA) = Bk^\alpha(\sA)$
We observe that the linear orderings of Hausdorff rank 0 are exactly $\eta$, $1+\eta$, $\eta+1$, $1+\eta +1$, and $1$.
From this perspective, $sp-$homogeneous orderings cannot be too complicated.

\begin{proposition} 
Every weakly $sp-$homogeneous linear ordering $\sL$ has Hausdorff rank $\leq 2$. 
\end{proposition}

\begin{proof} By Corollary \ref{cor:BkHasFinSuccessors}, $Bk(\sL)$ has only finitely many successors.
This means that $Bk(\sL)$ is finite, or it is a finite sum of finite orderings and copies of $\eta$.
In the former case, $Bk^2(\sL)=1$.
In the latter case, say $Bk(\sL)=n_1+\eta+n_2+\cdots+\eta+n_k$.
Then $Bk^2(\sL)= \eta+1+\cdots+1+\eta\cong\eta$ possibly with added end points.
Therefore, $Bk^2(\sL)$ has rank $0$, and $\sL$ has rank at most $2$.
\end{proof}

\section{Index Sets} \label{sec:index}

Our goal in this section is to calculate the index set of the computable $sp$-homogeneous and weakly $sp$-homogeneous linear orderings.
This is done using entirely hands-on, effective methods.
We prove that the set of $sp$-homogeneous linear orderings is $\Pi_5^0$ complete while the set of weakly of $sp$-homogeneous linear orderings is $\Sigma_6^0$ complete.
We revisit this type of problem in Section \ref{sec:Boldface} using the methods of descriptive set theory to obtain a boldface analog of these results.
We delay any discussion of the more abstract tools used for those methods until reaching that section.
We begin with an argument about the complexity of isomorphisms between shuffle sums.

We note here that there are some $sp$-homogeneous linear orderings that are computationally complicated.
We show below that there are two isomorphic computable shuffle sums that have no $\Delta_3$ isomorphism between them.
We study the complexity of isomorphisms between $sp$-homogeneous orderings in detail in Section 5.
That said, we present this argument here because the techniques are relevant to index set calculations.
The non-$\Delta_3$ isomorphism of certain $sp$-homogeneous structures turns out to be closely related to demonstrating the hardness of certain index sets of shuffle sums.

\begin{proposition} \label{shnotisom} 
There are two computable copies of $Sh(\omega + 1)$ which are not $\Delta^0_3$ isomorphic.
\end{proposition}

\begin{proof}  Let $L_0$ be a standard copy, where $\{x: |b(x)| = k\}$ is computable for each $k \leq \omega$. We will construct a copy $L$
such that  $\{x: |b(x)| = \omega\}$ is $\Pi^0_3$ complete. It follows that there is no $\Delta^0_3$ isomorphism from $L_0$ to $L$. The linear ordering $L$ is constructed as follows. The universe of $L$ may be viewed as a subset of $B \times \omega$,
where $B = Bk(L) = \bQ \cap (0,1)$.  We will define a computable reduction of the $\Sigma^0_3$ complete set $Cof = \{e: W_e\ \text{is cofinite}\}$ to $L$, so that 
the block $[f(e)]$ has finite type if and only if $e \in Cof$. 
Let $q_e = 2^{-e-1}$ and $a_e = \la q_e,0 \ra$. The reduction is defined so that $f(e) = a_e$. We will assume that at any stage, at most one element comes into any $W_e$. The construction is in stages as follows. 

At stage 0, each $[a_0] = \{a_0\}$ and in general $[\la q,0\ra] = \{\la q,0\ra\}$.

After stage $s$, we have, for each $e$, a value $n = n_{e,s}$ and a finite sequence $[a_e] = \{a_e, a^s_{e,1} < \dots < a^s_{e,n-1}\}$ where $n_{e,s}$ is our current guess at the least $n$ such that every $m>n$ belongs to $W_e$ and 
$n -1$ does not belong to $W_{e,s}$. Also we let $m_{e,s}$ be the largest such that $\{n_{e,s},n_{e,s}+1,\dots,m_{e,s}\}$ all belong to $W_{e,s}$.  For each interval $(q_{e+1},q_e)$, as well as the interval $(q_0,1)$, we have begun to construct a copy of $Sh(A)$. 

At stage $s+1$, we continue to build the copies of $Sh(A)$ on each interval. 
In other words, we add blocks of all sizes between each block that we have already built.
For the block $[a_e]$, we look to see if a new element $m$ has come into $W_e$ at stage $s+1$. If no element comes into $W_e$, then we extend the block $[a_e]$ by adding the next available element of the form $\la q_e, x \ra$.

There are four cases when a new element $m$ comes into $W_e$. 

(1) The first case is when $m = m_{e,s}+1$, giving us more evidence that $n_{e,s}$ might indeed be the (least) $n$ such that all $m \geq n$ belong to $W_e$. Then $n_{e,s+1} = n_{e,s}$ and we leave the block $[a_e]$ as it is. 

\medskip

(2) Next, suppose that $m = n_{e,s} - 1$. Then $n = n_{e,s+1}$ will be the least such that $n,n+1,\dots,m$ all belong to $W_{e,s+1}$. In this case, the block of $a_e$ becomes $\{a_e, a^s_{e,1} < \dots < a^s_{e,n-1}\}$. The discarded elements $a^s_{e,n}, \dots, a^s_{e,n_{e,s}}-1$ are sent to one of the blocks being built in the interval $(q_{e+1},q_e)$. 

\medskip 

(3) Next, suppose that $m < n_{e,s} - 1$.  Then $n = n_{e,s+1}$ will be the least such that $n,n+1,\dots,m$ all belong to $W_{e,s+1}$. So $m_{e,s+1}$ will be the largest $k$ such that $n,n+1,\dots,k$ all belong to $W_{e,s+1}$. The block of $a_e$ will now be $a_e,\dots,a^s_{e,n}-1$ and as in case (2) the discarded elements $a^s_{e,n}, \dots, a^s_{e,n_{e,s}}-1$ are put to work in $(q_{e+1},q_e)$.

\medskip

(4) Finally, suppose that $m > m_{e,s}+1$. Then we let $n_{e,s+1} = m$ and we let  $m_{e,s+1}$ be the largest such that 
$\{n_{e,s},n_{e,s}+1,\dots,m_{e,s+1}\}$ all belong to $W_{e,s}$. The block of $a_e$ at stage $s+1$ will now be 
$\{a_e, a^s_{e,1} < \dots < a^s_{e,n_{e,s}-1}\}$ together with the next $m - n_{e,s}$ available elements in $\{q_e\} \times \omega$. 

\medskip

Now we verify that this construction works as desired. 
We observe that if $n$ is not the least such that all $m \geq n$ belong to $W_e$, then $n = n_{e,s}$ only finitely many times. 

First, suppose that $W_e$ is cofinite and let $n$ be the least such that every $m\geq n$ belongs to $W_e$. Let $s$ be large enough so that $n_{e,t} \geq n$ 
for all $t \geq s$.  It follows from the construction that $[a_e] =  
\{a_e < a^s_{e,1} < \dots < a^s_{e,n-1}\}$.
In particular, once all of the elements less than $n$ are enumerated into $W_e$ and some next element in the interval above $n$ is enumerated, each of the elements $a_e < a^s_{e,1} < \dots < a^s_{e,n-1}$ forever stays inside of the block $[a_e]$.
Furthermore, consider any larger element added to the block $[a_e]$.
Eventually, some new element $m>n$ is enumerated into $W_e$ and is proposed as a new witness to cofiniteness.
This witness is eventually defeated when each of the elements in $[n,m]$ is enumerated into $W_e$, so eventually (either by the action in Case (2) or (3)) the larger elements are no longer in the block $[a_e]$.

Next, suppose that $W_e$ is not cofinite. Then for every $n$, we have $n = n_{e,s}$ only finitely often. Given $k \in \omega$, let $s$ be a stage such that, for any $t \geq s$, we never have $n_{e,t} < k$. Then, for all stages $t > s$, $[a_e]$ contains the initial sequence $\{a_e, a^s_{e,1}, \dots, a^s_{e,k}\}$. It follows that at the end of the construction,  $[a_e]$ will be a block of type $\omega$, 
that is, $\{a_e < a^s_{e,1} < \cdots\}$ form a block.
\end{proof}

We observe that the proof of this theorem shows that given any $\Sigma^0_3$ relation $R(x)$, we can uniformly and computably build shuffle sums $L(x)$ such that a fixed block in $L(x)$ is finite when $R(x)$, and has type $\omega$ when $\neg R(x)$. 
This will be needed below.

Some preliminary facts regarding index sets and linear orderings are needed before moving on to the proof of our main theorem for this section. Recall $x \sim y$ if $x$ and $y$ belong to the same block, that is, there is a sequence of successors leading $x$ to $y$ or from $y$ to $x$. We write $[x] < [y]$ to mean that $x < y$ and $\neg x \sim y$. 

\begin{proposition} \label{block1} For any computable linear ordering, and any $n\in\omega$:
\begin{enumerate}
\item $\{\la x,y \ra: y = succ(x)\}$ is $\Pi^0_1$;
\item $\{\la x,y\ra : x \sim y\}$ is $\Sigma^0_2$; 
\item $\{(x,y): [x] < [y]\}$ is $\Pi^0_2$.
\item $\{x: |[x]| \geq n\}$ is $\Sigma^0_2$;
\item $\{x: |[x]| = n\}$ is $D^0_2$, that is a conjunction of $\Sigma^0_2$ and $\Pi^0_2$ sets, and hence is $\Delta^0_3$;
\item $\{x: [x] \cong \omega\}$, $\{x: [x] \cong \omega^*\}$, and
  $\{x: |[x] \cong \zeta\}$ are all $\Pi^0_3$.
\end{enumerate}
\end{proposition}

\begin{proof} (1) $y = succ(x) \iff x < y\ \land (\forall z) \neg (x<z<y)$.

\medskip

(2) $x \sim y$ if there is a finite sequence $z_1 < z_2 < \dots < z_n$ 
such that each $z_{i+1} = succ(z_i)$ and either $x = z_1$ and $y = z_n$ or $x = z_n$ and $y = z_1$. 

(3) is immediate from (2). 

(4) follows from (1) directly by asserting the existence of $n-1$ successors or predecessors. 

(5) is immediate from (4)

(6) comes down to asserting the existence of arbitrarily long successor or predecessor chains from $x$ (which can be done in a $\Pi^0_3$ manner) and the non-existence of some finite length successor or predecessor chains starting from $x$ (which can be done in a $\Pi^0_2$ manner).
\end{proof}

As usual, let $\phi_e$ be the $e^{th}$
 partial computable function. We say that $\phi_e$ defines a linear ordering $L_e$ when $\phi_e$ is total and is the characteristic function of a relation $L_e \subset \omega \times \omega$.  
Now let $LI = \{e: \phi_e\ \text{defines a linear ordering}\}$.
The following are several basic completeness results that are needed for our final construction. 
Their proofs are included here as they are not listed fully in our desired form in the literature.

\begin{proposition} \label{Le1} Let $n>1$ be finite.
\begin{enumerate}
\item $LI$ is a $\Pi^0_1$ set.
\item $\{e: L_e \cong \eta\}$ is $\Pi^0_2$ complete.
\item  $\{e: L_e \cong n \cdot \eta\}$ is $\Pi^0_3$ complete.
\item $\{e \in LI: L_e \cong \omega \cdot \eta \}$,  $\{e \in LI: L_e \cong \omega^* \cdot \eta \}$, and $\{e \in LI: L_e \cong \zeta \cdot \eta\}$ are all $\Pi^0_4$ complete.
\item $\{e \in LI: L_e \cong \omega\}$,  $\{e \in LI: L_e \cong \omega^*\}$, and $\{e \in LI: L_e \cong \zeta\}$ are all $\Pi^0_3$ complete.
\end{enumerate}
\end{proposition}

\begin{proof} The upper bounds on the complexity are easy to see. 
We will just show that  $\{e: L_e \cong n \cdot \eta\}$ is $\Pi^0_3$.
That is, for $e \in LI$, $L_e \cong n \cdot \eta$ if and only if every block size is $n$, and 
\begin{itemize}
\item for all $x$, $|[x]| = n$;
\item for any $x$, there exist $y_1 < y_2 < y_3 < x$;
\item for any $x$, there exist $y_1 > y_2 > y_3 > x$;
\item for any $x_0 < x_1$, there exist $y_1,y_2,y_3$ such that 
$x < y_1 < y_2 < y_3 < x_1$. 
\end{itemize}

\medskip

For the completeness:

(2) We obtain a reduction $f$ from the $\Pi^0_2$ complete set $Inf = \{e: W_e\ \text{is infinite}\}$ so that $L_{f(e)} \cong \eta$ if $e \in Inf$,and $L_{f(e)} \cong \zeta$ if $e \notin Inf$. Initially, $L = L_{f(e)} = \{0\}$. After stage $s$, we will have a finite ordering $L^s = a_0^s < a_1^s < \dotsb < a_{n_s -1}$ of size $n_s$ with domain $n_s$.  At stage $s+1$, there are two cases.

\begin{itemize}
    \item If no new element is enumerated into $W_e$, then we add $n_s$ on the left of $L^s$,  and we add $n_s+1$ on the right. Thus $n_{s+1} = n_s+2$.
    \item If a new element is enumerated into $W_e$, then we add $n_s$ on the left of $L^s$,
    we add $n_s+i+1$ between $a_i^s$ and $a_{i+1}^s$, and we add $2n_s$ on the right. Thus $n_{s+1} = 2n_s+1$. 
\end{itemize}

We have that $L = L_{f(e)}$ is the union of the orders $L^s$. If $W_e$ is finite, then clearly $L_{f(e)} \cong \zeta$. If $W_e$ is infinite, then the construction ensures that $L$ is dense and without endpoints, so that $L_{f(e)} \cong \eta$.

(3) We use the $\Pi^0_2$ completeness of $Inf$. Given an arbitrary $\Pi^0_3$ predicate $Q(e) \equiv (\forall m) P(e,m)$, where $P$ is $\Sigma^0_2$, there is a c.e.\ relation $R$ such that $P(e,m)$ if and only if $\{x: R(e,m,x)\}$ is finite. Here we may assume that if $\neg Q(e)$, then $\neg P(e,m)$ for co-finitely many $m$. 

It suffices to define a computable function $\phi$ so that, for all $e$, $Q(e)$ if and only if $L_{\phi(e)} \cong n \cdot \eta$.

Then $L_{\phi(e)}$ is defined as the sum $L_0 + n + L_1 + n +  \dotsb$, where for each $m$, $L_m \cong \eta$ if $P(e,m)$ and $L_m \cong n\cdot \eta$ if $\neg P(e,m)$. Thus, if $Q(e)$, then each $L_m \cong n \cdot \eta$, so that $L_{\phi(e)} \cong n \cdot \eta$, and otherwise at least one of the $L_m \cong \eta$, so that $L_{\phi(e)}$ is not isomorphic to $n \cdot \eta$. 

For the construction of the component $L_m$, we are building a copy of $n \cdot \eta$, but every time a new element $x$ satisfies $R(e,m,x)$, we break up our size $n$ blocks $y_1 < y_2 < \dots < y_n$ into $n$ blocks of size $n$ by starting to build a copy of $n \cdot \eta$ between each successor pair and adding new elements to expand each $y_i$ to a block of size $n$. Thus, if $P(e,m)$, then this expansion only happens finitely often, so that $L_m \cong n \cdot \eta$. 
If $\neg P(e,m)$, then the expansion happens infinitely often, so that no elements of $L_m$ will have (permanent) successors, and thus $L_m \cong \eta$, so that $L \cong n \cdot \eta + n + n \cdot \eta + \dotsb + n + n \cdot \eta + n + \eta + n + \eta + \cdots\not\cong n\cdot\eta$.

\medskip

(4) Let $P(x) \equiv (\forall a)R(a,x)$, where $R$ is $\Sigma^0_3$. We use the technique from the proof of Proposition \ref{shnotisom} to build $L_{f(x)}$ as a dense family of blocks. Each block will have the form $M_0 + M_1 + \cdots$, where $M_a$ is used to check whether $R(a,x)$. The component $M_a$ will have finite size if $R(a,x)$, and it will have type $\omega$ if $\neg R(a,x)$. As in the proof of Proposition \ref{shnotisom}, we use the $\Sigma^0_3$ completeness of $Cof$ to define $M_a$ so that:

\begin{itemize}
\item 
 When $R(a,x)$, $M_a$ will have finite size $n_e$, where $n_e$ is the least $n$ such that all elements $\geq n$ belong to the appropriate c.e.\ set, and 

\item When $\neg R(a,x)$, $M_a$ will have type $\omega$. 
\end{itemize}

The discarded elements outlined in the proof of Proposition \ref{shnotisom} are placed individually in new blocks $M_i$ that have not yet been started. At each stage of the construction, in addition to updating the $M_a$, we start to build new blocks between any two blocks and add new blocks at each end, so that we will get a dense family of blocks. 

If $P(x)$, then each $M_a$ will be finite, so that $M_0 + M_1 + \cdots$ will have type $\omega$ and therefore $L_{f(x)} \cong \omega \cdot \eta$. 

If $\neg P(x)$, then eventually each $M_a$ will have type $\omega$,  so that $M_0 + M_1 + \cdots$ will have type $\omega \cdot \omega$ and therefore $L_{f(x)} \cong \omega^2 \cdot \eta$. 

The proof is very similar when considering $\omega^*\cdot\eta$ or $\zeta\cdot\eta$.

\medskip

(5) We again use the $\Pi^0_2$ completeness of $Inf$. Given an arbitrary $\Pi^0_3$ predicate $Q(e) \equiv (\forall m) P(e,m)$, where $P$ is $\Sigma^0_2$, there is a c.e.\ relation $R$ such that $P(e,m)$ if and only if $\{x: R(e,m,x)\}$ is finite. Here we may assume that if $\neg Q(e)$, then $\neg P(e,m)$ for co-finitely many $m$. Also, we may assume that at any stage $s+1$, there is exactly one pair $(m,x)$ such that $e,m,x$ comes into $R$ at stage $s+1$. 

It suffices to define a computable function $\phi$ so that, for all $e$, $Q(e)$ if and only if $L_{\phi(e)} \cong \omega$. 

The ordering $L_{\phi(e)}$ is defined as the sum $L_0 + 3 + L_1 + 3 +  \dotsb$, where for each $m$, $L_m \cong \omega$, if $\neg P(e,m)$, and $L_m$ is finite and has size $k_m$ if $\{x: P(e,m,x)\}$ has finite size $k_m$. 

For the construction, we start with each $L_m = \emptyset$.  At any stage $s+1$,
when $(e,m,x)$ enters $R$, we add $x$ to the end of $L_m$. 

The argument for $\omega^*$ and $\zeta$ is similar. 

\end{proof}

\begin{theorem} \begin{enumerate}
\item $\{L_e: L_e\ \text{is homogeneous}\}$ is $\Pi^0_2$ complete.
\item An ordering $\{e: L_e\ \text{is weakly homogeneous}\}$ is $\Sigma^0_3$ complete.
\end{enumerate}
\end{theorem} 

\begin{proof} (1) This follows from part (2) of Proposition \ref{Le1}, since a linear ordering $L$ is homogeneous if and only if it is isomorphic to $\eta$.

(2) $L_e$ is weakly homogeneous if and only if there exist $k$ and elements $x_0,\dots,x_{k-1}$ such that the intervals $(-\infty, x_1)$, $(x_i,x_{i+1})$ (for each $i$), and $(x_{k-1},\infty)$ are all either empty or are homogeneous. Therefore, the set is $\Sigma^0_3$.

In particular, note that, for any finite $k$, $\eta + k + \eta$ is weakly homogeneous, whereas $\eta + \omega + \eta$ is not weakly homogeneous. 

For the completeness, we will use the technique of Proposition \ref{shnotisom} to define a block $[a_e]$ that has finite type if $e \in Cof$, and otherwise has type $\omega$. Now we are simply building copies of $\eta$ on both sides of this block. Thus, when the construction requires discarding some elements at the end of $[a_e]$ at stage $s$, we simply move them into the copy of $\eta$, which we are building on the right. At the end of the construction, we have $L_{f(e)} = \eta + [a_e] + \eta$. It follows that $L_{f(e)}$ is weakly homogeneous if and only if $e \in Cof$. 
\end{proof} 

We note that $L$ is homogeneous if and only if it has no successors.
Also, $L$ is weakly homogeneous if and only if it has finitely many successors, and this is also if and only if $L$ is relatively computably categorical.
This is all explored in detail in \cite{AC17}.

\begin{proposition} For any $n \in \omega$: 
\begin{enumerate}
\item $\{e: |Bk(L_e)| \geq n\}$ is $\Sigma^0_3$;
\item $\{e: |Bk(L_e)| \leq n\}$ is $\Pi^0_3$; 
\item $\{e: |Bk(L_e)| = n\}$ is $D^0_3$; 
    \item $\{e: Bk(L_e) \text{is finite}\}$ is $\Sigma^0_4$; 
    \item $\{e: Bk(L_e) \cong \omega\}$ and $\{e: Bk(L_e) \cong \omega^*\}$ , and $\{e: Bk(L_e) \cong \zeta\}$ are all $\Pi^0_5$;
    \item $\{e: Bk(L_e) \cong \eta\}$ is $\Pi^0_4$; 
    \item $\{e: Bk(L_e) \cong n \cdot \eta\}$ is $\Pi^0_5$; 
    \item $\{e: Bk(L_e) \cong \omega \cdot \eta\}$ is $\Pi^0_4$. 
\end{enumerate}
\end{proposition}

\begin{proof} These all follow from the proof of Proposition \ref{Le1} together with Proposition \ref{block1}, by replacing each $x<y$ with $[x] < [y]$. 
\end{proof}

We say that the linear ordering $L$ is \emph{strongly $\eta$-like} if $Bk(L) \cong \eta$ and there is a finite bound on the size of the blocks. We say that $L$ is $\eta$-like if $Bk(L) \cong \eta$ and all blocks are finite.

\begin{theorem} \label{shornot} \begin{enumerate} 

    \item $\{e: L_e$ is strongly $\eta$-like$\}$ is $\Sigma^0_3$ complete. 
    \item $\{e: L_e$ is $\eta$-like$\}$ is $\Pi^0_4$ complete. 

        \item $\{e: L_e \cong Sh(A) \ \text{for some $A \subseteq \omega$}\}$ is $\Pi^0_4$ complete.

    \item $\{e: L_e \cong Sh(A) \ \text{for some $A \subseteq \omega\cup \{\omega, \omega^*,\zeta\}$}\}$ is $\Pi^0_5$ complete.

\end{enumerate}
\end{theorem}

\begin{proof} First, we note that if all blocks are finite, then the blocks must be densely ordered. Given that all blocks are finite, $L_e$ has no leftmost block if and only if it has no least element, which is a $\Pi^0_2$ condition.
Not having a rightmost block is similar. 
 
(1) For the upper bound, the definition of strongly $\eta$-like just requires that $(\exists b)(\forall x)|[x]| \leq b$, which is $\Sigma^0_3$ by Proposition \ref{block1}. 

For the completeness, we modify the argument used in Proposition \ref{shnotisom}. There we built a copy of $Sh(\omega+1)$ such that for each $e$, a single block $[b_e]$ had type $\omega$ when $e \notin Cof$, and had size $n_e$ when $n_e$ was the least $n$ such that every $m \geq n$ belongs to $W_e$. Now, we fix $e$ and build a linear ordering $L_{\phi(e)}$ that is a copy of $Sh(\omega)$ if $e \in Cof$, and is a copy of $Sh(\omega)+\omega+Sh(\omega)$ if $e \notin Cof$ in an analogous manner.  

%xxx

%alternative proof: 

%$L_{\phi(e)} \cong \omega \cdot \eta$ when $e \notin Cof$ and 
%$L_{\phi(e)} \cong n_e \cdot \eta$ when $e \in Cof$. Here is the construction for a fixed $e$.  At any stage $s$, we are building a dense set of blocks of size $n_s$, where $n_s$ is obtained as in the construction from Proposition \ref{Le1}. 
%When $n_{s+1} < n_s$, each discarded elements are now used as the initial element of some new blocks of size $n_{s+1}$. 

%Here is the verification. Suppose that $e \in Cof$ and let $n = n_e$ be the least $n$ such that every $m \geq n$ belongs to $W_e$. Then every element will belong to a block of size $n_e$, so the construction will produce a copy of $n \cdot \eta$.  If $e \notin Cof$, then every block eventually converges to an infinite sequence of successors, so that $L_{\phi(e)} \cong \omega \cdot \eta$. 

\medskip

We prove (2) and (3) together.
For the upper bound of (2), we just need to say that $(\forall x)(\exists k)|[x]| \leq k$, which is $\Pi^0_4$ by Proposition \ref{block1}. 
The upper bound for (3) is similar, but with the added statement that 
\[\Big((\forall x\neq y)\exists z_1 \exists z_2 \exists z_3 ~( [z_1]\cong k \land [z_2]\cong k \land [z_3]\cong k \land [z_1]<[x]<[z_2]<[y]<[z_3])\ \Big)\lor \forall w ~ [w]\not\cong k.\]
This is $\Pi_4^0$ by Proposition \ref{block1}.

For the completeness, let $P(x) \equiv (\forall a)R(a,x)$, where $R$ is $\Sigma^0_3$. We use the technique from the proof of Theorem \ref{shnotisom} to build $L_{f(x)}$ as a dense family of blocks. When $R(a,x)$, then the block $M_a$ will be finite, and when $\neg R(a,x)$, the block $M_a$ will have type $\omega$. We also add blocks of every finite size and now build a shuffle sum by ensuring that the blocks of every type are dense. 

If $P(x)$, then each $M_a$ will be finite, so that $L_{f(x)} \cong Sh(\omega)$.

If $\neg P(x)$, then some $M_a$ will have type $\omega$, $L_{f(x)} \cong Sh(\omega+1)$. 

\medskip

(4) For the upper bound, we must take for each block type $B$,
\[\Big((\forall x\neq y)\exists z_1 \exists z_2 \exists z_3 ~(  [z_1]\cong B \land ~ [z_2]\cong B \land  [z_3]\cong B \land [z_1]<[x]<[z_2]<[y]<[z_3])\Big)\lor \forall w ~ [w]\not\cong B.\]
This is $\Pi_5^0$ by Proposition \ref{block1}.

Let $P(x) \equiv (\forall a) R(a,x)$, where $R$ is $\Sigma^0_4$. 
From the uniform proof of parts (2) and (3), we may define $M_a = L_{f(x,a)}$ to be isomorphic to $Sh(\omega+1)$ if $R(a,x)$ and isomorphic to $Sh(\omega)$ if $\neg R(a,x)$.  If $P(x)$, then every $M_a \cong Sh(\omega+1)$, and if $\neg P(x)$, then, without loss of generality, $M_a \cong Sh(\omega+1)$ for $a \leq k$ and $M_a \cong Sh(\omega)$ for all $a > k$.  Now let $L_{\phi(x)} = Sh(\omega+1) + M_0 + M_1 + \cdots$. 

If $P(x)$, then $L_{\phi(x)}  \cong Sh(\omega+1) + 
Sh(\omega+1) + \cdots \cong Sh(\omega+1)$, which is a shuffle sum. 

If $\neg P(x)$, then  $L_{\phi(x)} = Sh(\omega+1) + \dots +  Sh(\omega+1) + Sh(\omega) + Sh(\omega) + \dots \cong Sh(\omega+1) + Sh(\omega)$, which is not a shuffle sum. 
\end{proof}

\begin{theorem} \label{homcplx} \begin{enumerate}
    \item $\{e: L_e \ \text{is $sp$-homogeneous}\}$ is $\Pi^0_5$ complete.
    \item $\{e: L_e \ \text{is weakly $sp$-homogeneous}\}$ is $\Sigma^0_6$ complete.  
\end{enumerate}
\end{theorem}

\begin{proof}  The $\Pi^0_5$ definition says the following. For any element $x$, 
either there is no other block isomorphic to $[x]$ or there is another block of type $[x]$, and for any three elements $y_1 < y_2$ and $z$, if each of them is between two blocks of type $[x]$, then there is a block of size $[z]$ between $y_1$ and $y_2$. 
This is of the desired complexity by the calculations in Proposition \ref{block1}.
It provides the correct definition because of Theorem \ref{thm:homogChar}.

(1) The completeness follows from the proof of parts (2) and (3) of Theorem \ref{shornot}. That is, if $P(x)$, then $L_{\phi(x)} \cong Sh(\omega+1)$, which is a shuffle sum, and therefore $sp$-homogeneous. If $\neg P(x)$, then $L_{\phi(x)} = Sh(\omega+1) + \dots +  Sh(\omega+1) + Sh(\omega) + Sh(\omega) + \cdots$, which is not a shuffle sum and is also not $sp$-homogeneous, since the two shuffle sum types are not disjoint. That is, consider two pairs $x_1 < y_1$ from the first $Sh(\omega+1)$ and $x_2 < y_2$ from the first $Sh(\omega)$, with $[x_1] < [y_1]$ and $[x_2] < [y_2]$. The map taking $[x_1]$ to $[x_2]$ and $[y_1]$ to $[y_2]$ is an isomorphism, but cannot be extended to an automorphism since there is a block of type $\omega$ between $x_1$ and $y_1$, but there is no such block between $x_2$ and $y_2$.  

(2) The $\Sigma^0_6$ definition says that there exist $x_0,\dots,x_{k-1}$ such that the intervals $(-\infty, x_1)_B$, $(x_i,x_{i+1})_B$ (for each $i$), and $(x_{k-1},\infty)_B$ are all either empty or are $sp$-homogeneous. A straightforward modification of the definition of $sp$-homogeneous yields the desired complexity.
It provides the correct definition because of Theorem \ref{thm:homogChar}.

For the completeness, let $P(x) \equiv (\exists a)R(a,x)$, where $R$ is $\Pi^0_5$.
It follows from a straightforward modification of the proof of part (4) of Theorem \ref{shornot} that we can build linear orderings $M_{a,x}$ which are isomorphic to $Sh(\omega+1)+1+Sh(\omega+1)\cong Sh(\omega+1)$ if $\neg R(a,x)$ and isomorphic to $Sh(\omega+1)+1+Sh(\omega)$ if $R(a,x)$.  Then let $L_{f(x)} \cong M_{0,x} + M_{1,x} + \dotsb$. 

If $P(x)$, then, $L_{f(x)}  \cong Sh(\omega+1) + 1 + Sh(\omega)  + 1 + \cdots + Sh(\omega) + 1 + Sh(\omega+1)$. This is weakly $sp$-homogeneous, where the exceptional points are those in the finitely many blocks written as explicit 1s in the above presentation. 

If $\neg P(x)$, then  $L_{f(x)}  \cong Sh(\omega+1)+1+Sh(\omega) + 1 + \cdots +Sh(\omega+1)+1+Sh(\omega) + \cdots$. This is not weakly $sp$-homogeneous by the following argument.  Given any finite set of exceptional points, let $a$ be large enough so that all of them are below the block $M_{a,x}$.  Now the remainder of $L_{f(x)} \cong Sh(\omega+1)+1+Sh(\omega) + 1 + \cdots +Sh(\omega+1)+1+Sh(\omega) + \cdots$.
Let $[x]$ and $[y]$ be finite blocks of the same size, one in a copy of $Sh(\omega)$ and the other in a copy of $Sh(\omega+1)$.
These generate isomorphic substructures that are not automorphic.
Therefore, $L_{f(x)}$ is not weakly $sp$-homogeneous.
\end{proof}

\section{Boldface Completeness}\label{sec:Boldface}

This section is dedicated to proving the boldface version of the results seen in Section \ref{sec:index}.
Instead of index sets for computable orderings, we work directly with presentations of countable linear orderings in the Polish space $Mod(LO)$.
While uniform relativization of the approach in Section \ref{sec:index} represents a possible approach to prove the theorems in this section, we instead offer a development more native to the descriptive set-theoretic viewpoint.
Our main tool is the boldface pair of structures theorem, a descriptive set-theoretic corollary of the classical result of Ash and Knight \cite{AK90}.
This approach has been taken before, see for example \cite{HMapprox} Theorem 2.2 or Section 5 of \cite{CGHT}.
The key to this approach is that there is a connection between boldface hardness theorems at the standard, asymmetric back-and-forth relations.

\begin{definition}
The \textit{standard asymmetric back-and-forth relations} $\leq_\alpha$, for a countable ordinal $\alpha < \omega_1$, are defined by the following:
    \begin{itemize}
        \item $(\mathcal{M},\bar{a}) \leq_0 (\mathcal{N},\bar{b})$ if $\bar{a}$ and $\bar{b}$ satisfy the same quantifier-free formulas from among the first $|\bar{a}|$-many formulas.
        \item For $\alpha > 0$, $(\mathcal{M},\bar{a}) \leq_\alpha (\mathcal{N},\bar{b})$ if for each $\beta < \alpha$ and $\bar{d} \in \mathcal{N}$ there is $\bar{c} \in \mathcal{M}$ such that $(\mathcal{N},\bar{b} \bar{d}) \leq_\beta (\mathcal{M},\bar{a} \bar{c})$.
    \end{itemize}
We define $\bar{a} \equiv_\alpha \bar{b}$ if $\bar{a} \leq_\alpha \bar{b}$ and $\bar{b} \leq_\alpha \bar{a}$.
\end{definition}

These relations are usually conceptualized as a game between the Spoiler, who moves first and picks a tuple of elements in alternating structures, and the Duplicator, who moves in response to this choice of tuple in the other structure.
The Spoiler attempts to copy the behavior exhibited by the Spoiler's choices and wins if they can do this until an ordinal clock expires.
This is useful for our purposes because of the aforementioned pair of structures theorem.

\begin{theorem}[\cite{AK90}]\label{thm:bfPoS}
    $\A\leq_\alpha\B$ if and only if $(Copies(\A),Copies(\B))$ is $(\mathbf{\Sigma}_\alpha^0,\mathbf{\Pi}_\alpha^0)-$hard.
\end{theorem}

This means that we can arrive at our desired result by carefully considering the back-and-forth relations between selected countable linear orderings.
The following are fundamental and extremely useful facts about calculating the back-and-forth relations for linear orderings seen in Lemma 15.8 of \cite{AK00}.

\begin{lemma}[\cite{AK00}]
Suppose $\A, \B$ are linear orders. Then $\A\leq_1 \B$ if and only if $\A$ is infinite or at least as large as $\B$. 
 
For $\alpha>1$, $\A\leq_\alpha \B$ if and only if for every $1\leq\beta<\alpha$ and every partition of $\B$ into intervals $\B_0,\dots,\B_n$, with endpoints in $\B$, there is a partition of $\A$ into intervals $\A_1,\dots,\A_n$ with endpoints in $\A$, such that $\B_i\leq_\beta \A_i$.
\end{lemma}

This lemma is so fundamental to our calculations that we often appeal to it without citation.
It should also be noted that, just like $\leq_1$ for linear orderings, $\leq_2$ is completely characterized and understood. 
This was first done by Montalb\'an in \cite{McountingBF}, see also \cite{GR23} for a more contemporary treatment.
It is a bit too complicated and, therefore, too much of a sidetrack to write out the full characterization here.
That said, we do remind the reader of the most important consequences of this result for our purposes.
\begin{theorem}\label{thm:2bnf}[\cite{McountingBF} Section 4.1]
    Say that $\L$ and $\K$ are linear orderings with arbitrarily long successor chains or that $\L$ convexly embeds into $\K$. Write $\L=\A_\L+\L'+\B_\L$ and $\K=\A_\K+\K'+\B_\K$ where
    \begin{itemize}
        \item $\A_\L$ (and analogously $\A_\K$) is empty if $\L$ has no first element and contains exactly the block of the first element if there is one,
        \item $\B_\L$ (and analogously $\B_\K$) is empty if $\L$ has no last element and contains exactly the block of the last element if there is one.
    \end{itemize}
    If $|\A_\L|\leq |\A_\K|$ and $|\B_\L|\leq |\B_\K|$, then
    \[\L\geq_2\K.\]
\end{theorem}
We appeal to this result to simplify calculations in this and later sections.
That said, we very occasionally also apply the more general form of the theorem from \cite{McountingBF} that includes the more complicated case without arbitrarily long successor chains.
In short, when querying the $\leq_2$ relation between linear orderings in a proof, we will generally cite the theorem above.
From there, we will state the answer that this work gives for the linear orderings in question.
This will be done instead of unraveling the full argument that is already provided in this source and others.

We begin by comparing the following two closely related shuffle sums.

\begin{lemma}\label{lem:4above}
\[
Sh(\omega+1)\leq_4 Sh(\omega).
\]
\end{lemma}

\begin{proof}

We provide a winning strategy for Duplicator in this game.
Let $\iota:Sh(\omega)\to Sh(\omega+1)$ be the natural embedding whose image is all of the blocks of finite size.
Say that Spoiler plays $\bar{p}$, which, without loss of generality, saturates the blocks that it intersects.
Duplicator can respond with the play of $\iota(\bar{p})$.
Note that the non-empty intervals defined by $\bar{p}$ are all isomorphic to $Sh(\omega)$.
Similarly, the non-empty itervals defined by $\iota(\bar{p})$ are all isomorphic to $Sh(\omega+1)$.
Therefore, this first move wins if and only if 
\[Sh(\omega) \leq_3 Sh(\omega+1).\]
Say that Spoiler plays $\bar{q}_1<\cdots<\bar{q}_n$ as their first move in this new game, where each $\bar{q}_i$ is a set of successors.
Without loss of generality, if $\bar{q}_i$ is in a finite block, it saturates this block, and if it is in a block of order type $\omega$, it is an initial segment of this copy of $\omega$.
In response to this move, Duplicator plays $\bar{r}_1<\cdots<\bar{r}_n$ increasing saturated finite blocks with $|\bar{q}_i|=|\bar{r}_i|$.
If $\bar{q}_i$ is from a finite block, the interval after it has order type $Sh(\omega+1)$; if $\bar{q}_i$ is from an $\omega$ block, the interval after it has order type $\omega+Sh(\omega+1)$.
Meanwhile, the non-empty intervals defined by $\bar{r}_1<\cdots<\bar{r}_n$ are all isomorphic to $Sh(\omega)$.
Therefore, it is enough to show that 
\[Sh(\omega) \geq_2 Sh(\omega+1) \text{  and  } Sh(\omega) \geq_2 \omega+ Sh(\omega+1). \]
These facts both follow from Montalb\'an's characterization of the 2 back-and-forth relations stated in Theorem \ref{thm:2bnf}.
One can also concretely check this fact by playing with the Duplicator along the previously mentioned embedding $\iota:Sh(\omega)\to Sh(\omega+1)$ and checking the cardinalities of the resulting intervals.

\end{proof}

In parallel to our lightface results, we now seek to compare two closely related linear orderings, one of which is $sp-$homogeneous and the other of which is not.

\begin{proposition}\label{prop:5below}
\[Sh(\omega+1)+Sh(\omega) \leq_5 Sh(\omega+1) \]
\end{proposition}

\begin{proof}
We describe a winning strategy for Duplicator in the back-and-forth game.
Let 
\[\nu: Sh(\omega+1)\to Sh(\omega+1)+Sh(\omega)\]
be the natural initial embedding.
Say that Spoiler plays the tuple $\bar{a}$ on their first move.
Duplicator responds with the play $\nu(\bar{a})$.
The intervals defined by $\bar{a}$ are isomorphic to those defined by $\nu(\bar{a})$ in every case but the end segment of the linear ordering.
If the largest value of $\bar{a}$ is within an $\omega$-block, its end segment is isomorphic to $\omega+Sh(\omega+1)$; in this case, the end segment in the other ordering is isomorphic to $\omega+Sh(\omega+1)+Sh(\omega)$.
If the largest value of $\bar{a}$ is not in an $\omega$ block, its end segment is isomorphic to $n+Sh(\omega+1)$ for some $n\in\omega$; in this case, the end segment in the other ordering is isomorphic to $n+Sh(\omega+1)+Sh(\omega)$.
Therefore, this is a winning move if the following relations hold for all $n$:
\[n+Sh(\omega+1)\cong n+Sh(\omega+1)+Sh(\omega+1)\leq_4 n+Sh(\omega+1)+Sh(\omega)
 \text{  and  } 
\]
\[
\omega+Sh(\omega+1)\cong \omega+Sh(\omega+1)+Sh(\omega+1)\leq_4 \omega+Sh(\omega+1)+Sh(\omega).
\]
Playing the back-and-forth game summand by summand and noting the isomorphisms between the first two summands yields that Duplicator has made a winning move, so long as
\[
Sh(\omega+1)\leq_4 Sh(\omega).
\]
This is exactly given by Lemma \ref{lem:4above}.

\end{proof}

\begin{corollary}\label{cor:5hard}
The set $([Sh(\omega+1)]_\cong,[Sh(\omega+1)+Sh(\omega)]_\cong)$ is $(\mathbf{\Pi}_5^0,\mathbf{\Sigma}_5^0)$ hard.
\end{corollary}

\begin{proof}
This follows from Proposition \ref{prop:5below} and Theorem \ref{thm:bfPoS}.
\end{proof}

We now analyze a slight variation of the above proposition that will be useful later.

\begin{proposition}\label{prop:5belowV2}
\[Sh(\omega+1)+1+Sh(\omega)+1 \leq_5 Sh(\omega+1)+1 \]
\end{proposition}

\begin{proof}
We describe a winning strategy for Duplicator in the back-and-forth game.
Let 
\[\nu: Sh(\omega+1)\to Sh(\omega+1)+1+Sh(\omega)+1\]
be the natural initial embedding.
Say that Spoiler plays the tuple $\bar{a}$, which without loss of generality includes the largest element $c$, on their first move.
Duplicator responds with the play $\nu(\bar{a})$ for the elements that are not the largest and the largest element $d$ in response $c$.
The intervals defined by $\bar{a}^\frown c$ are isomorphic to those defined by $\nu(\bar{a})^\frown d$ in every case but the segment between the largest element in $\bar{a}$ (resp. $\nu(\bar{a})$) and $c$ (resp. $d$).
If the largest value of $\bar{a}$ is within an $\omega$-block, this segment is isomorphic to $\omega+Sh(\omega+1)$; in this case, the segment in the other ordering is isomorphic to $\omega+Sh(\omega+1)+Sh(\omega)$.
If the largest value of $\bar{a}$ is not in an $\omega$ block, this segment is isomorphic to $n+Sh(\omega+1)$ for some $n\in\omega$; in this case, the segment in the other ordering is isomorphic to $n+Sh(\omega+1)+Sh(\omega)$.
The rest of the argument follows exactly as in Proposition \ref{prop:5below}.

\end{proof}

\begin{corollary}\label{cor:5hardV2}
The set $([Sh(\omega+1)+1]_\cong,[Sh(\omega+1)+1+Sh(\omega)+1]_\cong)$ is $(\mathbf{\Pi}_5^0,\mathbf{\Sigma}_5^0)$ hard.
\end{corollary}

\begin{proof}
This follows from Proposition \ref{prop:5belowV2} and Theorem \ref{thm:bfPoS}.
\end{proof}

\begin{theorem}
The set of linear orderings $(L,<)$ such that $(L,<,s,p)$ is homogeneous is $\mathbf{\Pi}^0_5$-hard.
\end{theorem}

\begin{proof}
This is witnessed by the reduction given in Corollary \ref{cor:5hard}.
In particular, note that $Sh(\omega+1)$ is homogeneous as a $\{<,s,p\}$ structure while $Sh(\omega+1)+Sh(\omega)$ is not because the finite blocks in $Sh(\omega+1)$ are not automorphic to those of the same size appearing in $Sh(\omega)$.
\end{proof}

We now prove a short, technical lemma that will allow us to upgrade our previous theorem to a statement about $\mathbf{\Sigma}^0_6$ and weakly $sp$-homogeneous linear orderings.

\begin{lemma}\label{lem:LOsumCont}
Given a sequence of continuous functions $f_i:X\to Mod(LO)$ reducing $(A_i,B_i)$ to $(C_i,D_i)$ and a countable linear ordering $\mathcal{L}$, there is a continuous map $F_\mathcal{L}:X\to Mod(LO)$ such that
\[F(x) = \sum_{i\in\mathcal{L}} f_i(x). \]
\end{lemma}

\begin{proof}
$F$ can be explicitly constructed out of the functions $f_i$.
Let
\[
F(x)(j) = \begin{cases}
f_i(x)(\langle k, k'\rangle ) \text{ if }  j=\langle \langle i,k \rangle,  \langle i,k' \rangle \rangle \\
1 \text{ if } j= \langle \langle \ell,k \rangle,  \langle \ell',k' \rangle \rangle \text{ and } \ell<_{\mathcal{L}} \ell' \\
0 \text{ if } j= \langle \langle \ell,k \rangle,  \langle \ell',k' \rangle \rangle \text{ and } \ell>_{\mathcal{L}} \ell'. 
\end{cases}
\]
It is straightforward to confirm that this map is continuous, as all of the $f_i$ are continuous.
Furthermore, it is explicitly constructed to have the desired property that $F(x) = \sum_{i\in\mathcal{L}} f(x)$.
\end{proof}

\begin{theorem}
The set of linear orderings $(L,<)$ such that $(L,<,s,p)$ is weakly homogeneous is $\mathbf{\Sigma}_6$ hard.
\end{theorem}

\begin{proof}
Let $A\subset X$ be a $\mathbf{\Sigma}_6$ set.
Write $A=\bigcup_{i\in\omega} B_i$ such that each $B_i$ is $\mathbf{\Pi}_5^0$.
Without loss of generality, $B_i\subseteq B_{i+1}$.
Let $f_{i}$ be the reduction witnessing this hardness for the set $(B_i,B_i^c)$ from Corollary \ref{prop:5belowV2}.
Let $\mathcal{L}$ be the standard copy of $\omega$ given by $0<1<2<3<\cdots$.
Apply Lemma \ref{lem:LOsumCont} to the sequence $f_i$ and the linear ordering $\mathcal{L}$.

Let $x\in A$.
In this case, $x\in B_a$ for some least $a$.
This means that $x\in B_i$ exactly for each $i\geq a$.
Therefore, $f_i(x)\cong Sh(\omega+1)+1$ for each $i\geq a$ while $f_i(x)\cong Sh(\omega+1)+1+Sh(\omega)+1$ for each $i< a$.
This gives that
\[
F(x) = \sum_{i\in\mathcal{L}} f_i(x) \cong\sum_{i<a}\Big(Sh(\omega+1)+1+Sh(\omega)+1\Big) +\sum_{i\geq a}\Big(Sh(\omega+1)+1\Big)\cong
\]
\[
(Sh(\omega+1)+1+Sh(\omega)+1)\cdot a + (Sh(\omega+1)+1)\cdot\omega \cong (Sh(\omega+1)+1+Sh(\omega)+1)\cdot a + Sh(\omega+1).
\]
Name the parameters given by each of the $2a$ many "1"s that appear in $(Sh(\omega+1)+1+Sh(\omega)+1)\cdot a$.
It is enough to show that each of the intervals between these parameters is homogeneous as a $\{<,s,p\}$ structure to demonstrate that this ordering is weakly homogeneous as a $\{<,s,p\}$ structure.
However, this is immediate, as each such interval is given by a shuffle sum.

Let $x\not\in A$.
This means that $x\not\in B_i$ for each $i$.
Therefore, $f_i(x)\cong Sh(\omega+1)+1+Sh(\omega)+1$ for each $i$.
This gives that
\[
F(x) = \sum_{i\in\mathcal{L}} f_i(x) \cong \sum_{i\in\omega} Sh(\omega+1)+1+Sh(\omega)+1 \cong (Sh(\omega+1)+1+Sh(\omega)+1)\cdot\omega.
\]
Given any finite tuple of parameters, the set of elements greater than all of them in $F(x)$ has infinitely many summands isomorphic to $Sh(\omega+1)$ and infinitely many alternating summands isomorphic to $Sh(\omega)$.
It is direct to show that a finite block in one of these $Sh(\omega)$ summands is not automorphic to the same-sized finite block in a $Sh(\omega+1)$ summand.
Therefore, $F(x)$ is not weakly $sp-$homogeneous, and we have constructed our desired reduction.
\end{proof}

Here are boldface versions of the simpler results from Section \ref{sec:index}.
In particular, the boldface approach is just as robust as the lightface analysis from that section.
We work in the space of relations on $\omega$, which is essentially $2^{\bN}$.

\begin{proposition} \label{Le1B} Let $n>1$ be finite.
\begin{enumerate}
\item $\{L: \text{$L$ is a linear ordering}\}$ is a closed set.
\item $Copies(\eta)$ is $\mathbf{\Pi}^0_2$ complete.
\item  $Copies(n \cdot \eta)$ is $\mathbf{\Pi}^0_3$ complete.
\item $Copies( \omega \cdot \eta )$,  $Copies( \omega^* \cdot \eta )$, and $Copies( \zeta \cdot \eta )$ are all $\mathbf{\Pi}^0_4$ complete.
\item $Copies(\omega)$, $Copies(\omega^*)$, and $Copies(\zeta)$ are all $\mathbf{\Pi}^0_3$ complete.
\end{enumerate}
\end{proposition}

\begin{proof} The upper bounds on the complexity are easy to see.

\medskip

For the completeness:

(2), (3), and (5) can all be seen as corollaries of Gonzalez and Rossegger's characterization of Scott complexities under $\Sigma_4$ for linear orderings (see \cite{GR23}).
That said, we also provide more direct arguments here.

(2) It is enough to observe that $\eta\geq_2\eta+1$ by the characterization of Montalb\'an stated in Theorem \ref{thm:2bnf} and apply Theorem \ref{thm:bfPoS}.

(3) We show that $n\cdot\eta\geq_3Sh(n,n-1)$ and apply Theorem \ref{thm:bfPoS}.
On the first move of the game, we may assume that Spoiler plays a sequence of entire $n$ blocks $\bar{p}_1,\cdots,\bar{p}_n$.
Duplicator may play along the natural embedding $\iota:n\cdot\eta\to Sh(n,n-1)$ whose image is the blocks of size $n$.
The resulting intervals require that $n\cdot\eta\leq_2Sh(n,n-1)$ for this to be a winning move for Duplicator.
This relation is readily confirmed by the characterization of Montalb\'an \cite{McountingBF}.
Because these orderings have a bound on the size of their successor chains, we cannot directly apply Theorem \ref{thm:2bnf} as stated (though this case is still covered in Montalb\'an's more general work).
For a more concrete approach, it is easy to confirm that if the Spoiler plays an increasing sequence of finite blocks, the Duplicator can respond with a sequence of increasing tuples where the blocks of size $n-1$ are responded to with the first $n-1$ elements of a block of size $n$ and the blocks of size $n$ are responded to with a block of size $n$.

\medskip

(4) We only analyze $Copies( \omega \cdot \eta )$ as the other arguments are very similar.
We show that $\omega\cdot\eta\geq_4Sh(\omega,\omega\cdot2).$
Here, we write $Sh(\omega,\omega\cdot2)$ as shorthand for $Sh(\{\omega,\omega\cdot2\})$
On the first move of the game, the Spoiler plays initial segments of some $\omega$ blocks $\bar{p}_1,\ldots,\bar{p}_n$.
The Duplicator may play along the natural embedding $\iota:\omega\cdot\eta\to Sh(\omega,\omega\cdot2)$ whose image is the shuffled in $\omega$.
The resulting intervals require that $\omega+\omega\cdot\eta\leq_3\omega+Sh(\omega,\omega\cdot2)$ for this to be a winning move for Duplicator.
Say that on this move the Spoiler plays increasing initial segments of some $\omega$ blocks $\bar{q}_1,\dots,\bar{q}_n$.
Duplicator will play increasing initial segments of some $\omega$ blocks of the same size in response.
Analyzing intervals, it is required that both $\omega+\omega\cdot\eta\geq_2\omega$ and $\omega+\omega\cdot\eta\geq_2\omega+Sh(\omega,\omega\cdot2)$ for this to be a winning move for Duplicator.
Both of these relations follow from Theorem \ref{thm:2bnf}. 

\medskip

(5) Observe that $\omega\geq_3\omega\cdot2$, $\omega^*\geq_3\omega^*\cdot2$, and $\zeta\geq_3\zeta\cdot2$ all follow from the classical analysis of Ash (see e.g. \cite{cst1} Lemma II.38). Apply Theorem \ref{thm:bfPoS}.

\end{proof}

\section{Categoricity level of $sp-$homogeneous linear orderings}

In this section, we classify the \textit{sp}-homogeneous linear orderings by their level of relative categoricity.
In other words, for each \textit{sp}-homogeneous ordering, we find the least $n$ such that it is relatively $\Delta_n$ categorical based on easily checkable structural properties of the ordering.
We show that $sp$-homogeneous linear orderings are always relatively $\Delta_4$ categorical.
Because $\Delta_1$ and $\Delta_2$ relative categoricity for linear orderings is so explicitly understood, this comes down to distinguishing the $\Delta_3$ relatively computably categorical linear orderings from those that are only $\Delta_4$ relatively computably categorical.

We make use of the key result of \cite{AC17}, which shows that there is an interesting connection between weak homogeneity and computability-theoretic properties of a structure.

\begin{theorem}\label{thmAC}
\begin{enumerate}
\item If a structure $\A$ is weakly homogeneous, then it is relatively $\Delta^0_2$ categorical.
\item If a structure $\A$ is locally finite and weakly homogeneous, then it is relatively computably categorical.
\end{enumerate}
\end{theorem}

The functions $s$ and $p$ are computable in $2$ jumps of the structure (even though the partial functions that are undefined when the successor or predecessor does not exist will be partial computable in only one jump). By the above, any weakly homogeneous 
\textit{sp}-linear ordering must be relatively $\Delta^0_2$ categorical.
Note that this means the underlying linear ordering without the successor and predecessor functions is relatively $\Delta^0_4$ categorical, as follows. 

\begin{proposition} Suppose that $(L,s,p)$ is weakly homogeneous. Then $L$ is relatively $\Delta^0_4$ categorical.  
\end{proposition}

\begin{proof} Let $L_1$ and $L_2$ be isomorphic copies of $L$. Then $s$ and $p$ are computable in $L_i$ from $L_i^{''}$. Since $L$ is weakly homogeneous, an isomorphism between $L_1$ and $L_2$ may be computed from $(L_1,s,p)^{'}$ and $(L_2,s,p)^{'}$ and hence from $L_1^{'''}$ and $L_2^{'''}$. 
\end{proof}

Note that not every relatively $\Delta_4$ or even relatively $\Delta_3$ categorical linear ordering is weakly $sp$-homogeneous.
Consider the examples from Corollary \ref{wspcor}.
It follows from a result of Frolov and Zubkov \cite{FZ24} that these examples are relatively $\Delta_4$ categorical.
More than this general upper bound, a recent preprint of the first four authors shows that these $\zeta$ representations are often relatively $\Delta_3$ categorical.
This shows that, while weakly $sp$-homogeneous orderings represent a large class, they do not completely subsume the class of relatively $\Delta_4$ or even relatively $\Delta_3$ categorical linear orderings.

The rest of this section is dedicated to understanding when a weakly $sp$-homogeneous linear ordering achieves its maximal level of complexity.
We will make use of the following auxiliary definitions for useful relations.

\begin{definition}
    Let $P_n$ be a unary predicate that holds of the elements in a linear ordering with $n$ predecessors.
    Let $S_n$ be a unary predicate that holds of the elements in a linear ordering with $n$ successors.
    Let $Adj_n$ be a binary predicate that holds of pairs of elements with exactly $n$ elements in between them.
\end{definition}

Note that $P_n$ and $S_n$ are $\Sigma_2$ definable.
$Adj_n$ is $\dSinf{1}$ definable.

To show that an ordering $\L$ is uniformly, relatively $\Delta_3$ categorical, it is equivalent to show that there is a computable list of computable $\Sigma_3$ formulas that define the orbits of all of the tuples in $\L$.
Such a set is called a $\Sigma_3$ Scott family.
To show that an ordering $\L$ is relatively $\Delta_3$ categorical, it is equivalent to show that there is a computable list of computable $\Sigma_3$ formulas over some finite set of parameters $\bar{p}\in \L$ that define the orbits of all of the tuples in $\L$.
Such a set is called a $\Sigma_3$ Scott family over $\bar{p}$.
See \cite{cst2} Chapter VII for more details on Scott families.
We will work with Scott families instead of considering $\Delta_3$ maps between structures directly.

\begin{proposition}\label{prop:Delta3begin}
    The following $sp-$homogeneous linear orderings are uniformly, relatively $\Delta_3$ categorical:
    \begin{enumerate}
        \item Those with only finite blocks;
        \item Those with $\omega$ blocks, $\omega^*$ blocks and finite blocks of bounded size;
        \item Those with $\zeta$ blocks and finite blocks of bounded size.
    \end{enumerate}
\end{proposition}

\begin{proof}
Let $\L$ be an $sp$-homogeneous.
By Theorem \ref{thm:homogChar}, the orbit of a tuple $\bar{x}\in\L$ is determined by
\begin{enumerate}
    \item The order of $\bar{x}$,
    \item The types of blocks that $\bar{x}$ lie in, and
    \item The position of $\bar{x}$ within the blocks it intersects.
\end{enumerate}
Therefore, to describe the automorphism orbit in the ordering $\L$, it is enough to isolate the automorphism orbit within each block and identify the isomorphism type of the block that the tuple lies in.

In the case that $\L$ has only finite blocks, the automorphism orbit of an element $x$ with exactly $n$ successors and $m$ predecessors can be isolated by $P_m(x)\land S_n(x)\land \lnot P_{m+1}(x)\land \lnot S_{n+1}(x)$. 
As every element has a finite number of successors and predecessors, these formulas are as desired for every element.
To distinguish elements $a,b$ that lie in the same block from those that lie in different blocks, it is enough to say either $Adj_\ell(a,b)$ or $\bigwwedge_{\ell\in\omega}\lnot Adj_\ell(a,b)$ where $\ell$ is the number of elements between any element in the orbit of $a$ and any element in the orbit of $b$ within their block.

In the second case, say that the maximal finite block has size $N$.
Within finite blocks, the same formulas as above define automorphism orbits.
Within the block of type $\omega$, an element $x$ must say how many predecessors it has to isolate its automorphism orbit.
In particular, if it has exactly $k$ predecessors $P_k(x)\land S_N(x)\land \lnot P_{k+1}(x)$ isolates the automorphism orbit of $k$.
A dual argument in $\omega^*$ shows that if $x$ has exactly $k$ successors then $S_k(x)\land P_N(x)\land \lnot S_{k+1}(x)$ isolates the automorphism orbit of $k$.
Distinguishing pairs of elements in the same blocks from those in different blocks is given by the same procedure as the one in the above paragraph.

In the final case, say that the maximal finite block has size $N$.
Again, within finite blocks, the same formulas as above demonstrate the desired construction.
Within the blocks of type $\zeta$, a tuple must simply state its (necessarily finite) $Adj_n$ type along with $S_N(x)$ to enforce that it is not in any finite block.
Distinguishing pairs of elements in the same blocks from those in different blocks is given by the same procedure as the one in the above paragraph.

Taken together, we obtain $\Sigma_3$ Scott families in each of the described cases, yielding the desired result. 
\end{proof}

Note that we have begun by looking at uniform relative categoricity.
It will follow from an argument later in the section, based on Theorem VIII.4 of \cite{cst1}, that we can leverage this understanding to understand relative categoricity without the uniformity assumption.

To describe the distinction between the $\Delta_3$ and $\Delta_4$ cases, we first need to observe the following consequence of the classification given in Theorem \ref{thm:homogChar}.

\begin{proposition}\label{prop:CasesDef}
    If $\L$ is $sp-$homogeneous, one of the following holds:
    \begin{enumerate}
        \item $\L\cong\L'+Sh(S)$ for $S\subseteq\omega\cup\{\omega,\omega^*,\zeta\}$;
        \item $\L\cong\L'+B$ where $B\in \omega\cup\{\omega,\omega^*,\zeta\}$ and $B$ is a block in $\L$;
        \item $\L$ has finite blocks of arbitrarily long size in every end segment.
    \end{enumerate}
\end{proposition}

\begin{proof}
    Say that $\L$ is not in Case 1 or Case 2.
    Consider a block $B$ in $\L$.
    By Theorem \ref{thm:homogChar}, $B$ is either the unique block of its order type or lies in a shuffle sum including its order type.
    By our assumption, these are not the last elements of the ordering.
    Therefore, there is a block $B'$ greater than every block with order type $B$.
    In sum, every block in $\L$ has a block of a different order type greater than it.
    As there are only finitely many infinite block types, this means that every block in $\L$ has infinitely many finite block types greater than it.
    The fact that $\L$ is in Case 3 follows at once.
\end{proof}

Throughout this section, we say that $\L$ is ``in Case 1'', ``in Case 2'', or ``in Case 3'' if $\L$ falls in the corresponding case of the above theorem.
We can also apply the above analysis to $\L^*$ to understand the initial behavior of $\L$.

\begin{corollary}
        If $\L$ is $sp$-homogeneous, then one of the following holds:
    \begin{enumerate}
        \item $\L\cong Sh(S)+\L'$ for $S\subseteq\omega\cup\{\omega,\omega^*,\zeta\}$;
        \item $\L\cong B+\L'$, where $B\in \omega\cup\{\omega,\omega^*,\zeta\}$ and $B$ is a block in $\L$;
        \item $\L$ has finite blocks of arbitrarily long size in every initial segment.
    \end{enumerate}
\end{corollary}

We say that $\L$ is in ``Case 1*'', ``in Case 2*'', or ``in Case 3*'' if $\L^*$ falls in the corresponding case of the above corollary.

Proposition \ref{prop:Delta3begin} indicates that the only potential barriers to being $\Delta_3$ relatively categorical are the infinite blocks.
Indeed, Proposition \ref{prop:Delta3begin} gives a uniformly computable $\Sigma_3$ Scott family that defines the automorphism orbits of the elements in the finite blocks, regardless of the underlying $sp$-homogeneous linear ordering, so their positioning is never of serious concern.
A key idea in the proof is to analyze the infinite blocks in the linear ordering locally.
We look at each of the infinite blocks in the ordering (or the shuffle sum they happen to be a part of) along with the linear orderings to their right and left (analyzing which case it is in the above definition).
This yields a series of cases for local behavior, some of which will be allowable in a $\Delta_3$ relatively categorical and some of which will forbid the ordering from being $\Delta_3$ relatively categorical.
By exhausting these cases, we will arrive at the desired classification.

We make extensive use of the following tool and proposition for our negative results.
These definitions make use of the previously introduced back-and-forth relations.

\begin{definition}[{\cite[Section 17.4]{AK00}}]
We say that a tuple is $\alpha$-\textit{free} in the structure $\A$ if for every tuple $\bar{b}\in A$ and every $\beta<\alpha$, there are tuples $\bar{a}',\bar{b}'$ such that 
\[\bar{a}\bar{b}\leq_\beta\bar{a}'\bar{b}'\]
\[\bar{a}\not\leq_\alpha\bar{a}'.\]
\end{definition}

\begin{proposition}[{\cite[Section 17.4]{AK00}}]
    A tuple is $\alpha$-free if and only if its automorphism orbit cannot be defined by a $\Sigma_\alpha$ formula.
    In particular, if $\A$ contains an $\alpha$-free tuple, it is not uniformly relatively $\Delta_\alpha$ categorical.
\end{proposition}

We also make consistent use of the classification of the 2-back-and-forth relations for linear orderings explicitly laid out in Theorem \ref{thm:2bnf}.
This allows us to deliver our proofs in a far more parsimonious manner.

\begin{convention}
    For the sake of cleaner notation, in the following proposition and throughout the section, we write $Sh(I_0,S_0)$ to mean $Sh(\{I_0\}\cup S_0)$ and $Sh(I_1,I_2,S_1)$ to mean $Sh(\{I_1,I_2\}\cup S_1)$ where the $I_k$ are infinite block types and the $S_j$ are sets of finite block types.
\end{convention}

We begin by noting that some shuffle sums with infinite blocks are never allowed in a $\Delta_3$ relatively categorical ordering.

\begin{proposition}\label{prop:forbiddenShuffles}
    Let $S$ be an arbitrary subset of $\omega\cup\{\omega,\omega^*,\zeta\}$ and $T$ be an infinite subset of $\omega$. Let $\L$ be one of the following orderings:
    \begin{enumerate}[label=(\alph*)]
        \item $Sh(\omega,\zeta,S)$,
        \item $Sh(\omega^*,\zeta,S)$,
        \item $Sh(\omega,T)$,
        \item $Sh(\omega^*,T)$,
        \item $Sh(\zeta,T)$,
        \item $Sh(\omega,\omega^*,T)$.
    \end{enumerate}
    Rephrased, (a) and (b) are shuffle sums containing $\zeta$ and a different infinite block, while (c)-(f) are shuffle sums containing an infinite block alongside finite blocks of arbitrary size.
    Given any $\A$ and $\B$, $\A+\L+\B$ is not uniformly $\Delta_3$ relatively categorical.
\end{proposition}

\begin{proof}
    We will show that in each case there is a $3$-free element in $\L$.
    In each case, the proof of $3$-freeness is unaffected by the presence of $\A$ and $\B$ in the ordering, so this will yield the desired result.

    We now analyze Case (a); by symmetry, Case (b) will be the same argument.
    Fix an element $a$ lying in $Z$, a block isomorphic to $\zeta$.
    We claim that $a$ is $3$-free.
    Let $\bar{b}$ be the tuple selected by the Spoiler.
    Write $\bar{b}=\bar{b}_1\bar{b}_2\bar{b}_3$, where $\bar{b}_1$ contains elements below $Z$, $\bar{b}_2$ contains elements in $Z$, and $\bar{b}_3$ contains elements above $Z$.
    Without loss of generality, $\bar{b}_2$ is a finite interval of size $m$ including $a$.
    Say that $a$ is the $k^{th}$ element of $\bar{b}_2$.
    Let $a'$ be the $k^{th}$ element of an $\omega$ block, $W$ greater than all of the elements in $\bar{b}_1$ and less than all of the elements in  $\bar{b}_3$.
    Let $\bar{b}'=\bar{b}_1\bar{b}'_2\bar{b}_3$ where $\bar{b}'_2$ is the first $m$ elements of $W$.
    It is immediate that $\bar{a}\not\geq_3\bar{a}'$ as the two elements have a different number of predecessors. 
    We claim that $\bar{a}\bar{b}\leq_2\bar{a}'\bar{b}'$.
    Checking interval by interval, the only non-isomorphic corresponding intervals are the ones below $\bar{b}_2$ and $\bar{b}'_2$.
    Thus, we need only check that
    \[Sh(\omega,\zeta,S)\geq_2 Sh(\omega,\zeta,S)+\omega^*,\]
    which follows from Theorem \ref{thm:2bnf}.

    We now analyze Case (c).
    In fact, all other cases will follow from the same argument as the one for Case (c), so this is the only remaining case that is explicitly considered.
    Let $a$ be an element of an $\omega$ block $W$.
     We claim that $a$ is $3$-free.
    Let $\bar{b}$ be the tuple selected by the Spoiler.
    Write $\bar{b}=\bar{b}_1\bar{b}_2\bar{b}_3$ where $\bar{b}_1$ contains elements below $W$, $\bar{b}_2$ contains elements in $W$ and $\bar{b}_3$ contains elements above $W$.
    Without loss of generality, $\bar{b}_2$ is a finite interval of size $m$ including $a$.
    Say that $a$ is the $k^{th}$ element of $\bar{b}_2$.
    Let $a'$ be the $k^{th}$ element of a finite block of size greater than $m$, that we call $F$, greater than all of the elements in $\bar{b}_1$ and less than all of the elements in  $\bar{b}_3$.
    Let $\bar{b}'=\bar{b}_1\bar{b}'_2\bar{b}_3$ where $\bar{b}'_2$ is the first $m$ elements of $F$.
    It is immediate that $\bar{a}\not\geq_3\bar{a}'$ as the two elements have a different number of successors. 
    We claim that $\bar{a}\bar{b}\leq_2\bar{a}'\bar{b}'$.
    Checking interval by interval, the only non-isomorphic corresponding intervals are the ones below and above $\bar{b}_2$ and $\bar{b}'_2$.
    Thus, we need only check that
    \[Sh(\omega,T)\geq_2 Sh(\omega,T)+k_1,\]
    \[k_2+Sh(\omega,\zeta,S)\geq_2 \omega+Sh(\omega,\zeta,S),\]
    for some finite orderings $k_i$.
    Both of these follow from Theorem \ref{thm:2bnf}.
\end{proof}

Note that, when restricting our attention to shuffle sums, Proposition \ref{prop:forbiddenShuffles} provides a full converse to Proposition \ref{prop:Delta3begin}.
More explicitly, Proposition \ref{prop:Delta3begin} implies that shuffle sums containing only finite blocks, along with some selection of infinite blocks not including both $\zeta$ and another infinite block or just infinite blocks, always have uniformly computable $\Sigma_3$ definitions of their automorphism orbits.
On the other hand, shuffle sums not in this form (exactly those considered by Proposition \ref{prop:forbiddenShuffles}) fail to have such $\Sigma_3$ definitions no matter what, even if they are placed in the context of a larger linear ordering.

We now move to analyze the shuffle sums from Proposition \ref{prop:Delta3begin} in the context of a larger ordering.
Each of these is sometimes allowable to appear in a uniformly relatively $\Delta_3$ categorical linear ordering, but sometimes they are not, depending on the local context of the ordering.
We begin with the case where the infinite block(s) are shuffled in with some \textit{non-empty} finite set of finite blocks.
This case is in some ways the ``nicest''; it does not come with any barriers to uniform relative $\Delta_3$ categoricity.
This will follow from the following proposition.

\begin{proposition}\label{prop:Sigma3blockabove}
    There is a uniformly computable set of $\Sigma_3$ formulas $\varphi_{n,<}$ and $\varphi_{n,>}$ which hold of exactly the points that are (respectively) below some block of size $n$ and above some block of size $n$.
\end{proposition}
\begin{proof}
    We write out the definitions explicitly:
    \[\varphi_{n,<}(y):= \exists x_1\dots, x_n<y\ \bigwedge_{j<n} Adj(x_j,x_{j+1})\land \forall z \lnot Adj(z,x_1)\land \lnot Adj(x_n,z)\]\
    
    \[\varphi_{n,>}(y):=\exists x_1,\dots,x_n>y \bigwedge_{j<n} Adj(x_j,x_{j+1})\land \forall z \lnot Adj(z,x_1)\land \lnot Adj(x_n,z)).\]
    
    \end{proof}

\begin{corollary}
        Let $S$ be a finite non-empty subset of $\omega$.
    Let $\L$ be one of the following orderings:
    \begin{enumerate}
        \item $Sh(\omega,S)$,
        \item $Sh(\omega^*,S)$,
        \item $Sh(\zeta,S)$,
        \item $Sh(\omega,\omega^*,S)$.
    \end{enumerate}
    Assume that $\A+\L+\B$ is $sp$-homogeneous. The automorphism orbits of the tuples of $\L$ within $\A+\L+\B$ are always definable by a uniform set of uniformly computable $\Sigma_3$ formulas.
\end{corollary}

\begin{proof}
    Because $\A+\L+\B$ is $sp$-homogeneous, the finite block sizes in $S$ do not appear anywhere in $\A$ or $\B$.
    Let $n\in S$.
    The elements of $\L$ are then definable by $\varphi_{n,<}(y)$ and $\varphi_{n,>}(y)$. 
    Within $\L$, the definitions from Proposition \ref{prop:Delta3begin} define the automorphism orbit; taking a conjunction of these definitions with $\varphi_{n,<}(y)$ and $\varphi_{n,>}(y)$ yields the desired result.
\end{proof}

\begin{corollary}\label{cor:defBoundedSummand}
    Let $\A$ and $\B$ be one of the following orderings:
    \begin{enumerate}
        \item $n\in\omega$,
        \item $Sh(S)$,
        \item $Sh(\omega,S)$,
        \item $Sh(\omega^*,S)$,
        \item $Sh(\zeta,S)$,
        \item $Sh(\omega,\omega^*,S)$,
    \end{enumerate}
    where $S$ is a finite non-empty subset of $\omega$.
    If $\A'+\A+\L+\B+\B'$ is $sp$-homogeneous, then $\A+\L+\B$ is $\Sigma_3$ definable in $\A'+\A+\L+\B+\B'$ by a computable formula.
\end{corollary}

\begin{proof}
    Let $\A$ be associated with the finite set $S_1$ and $\B$ with the finite set $S_2$ with block sizes $n_1\in S_1$ and $n_2\in S_2$. Then
    $\A+\L+\B$ is definable by $\varphi_{n_1,<}(y)$ and $\varphi_{n_2,>}(y)$.
\end{proof}

We now analyze $\omega$ as an individual block.
This analysis will be emblematic of all of the remaining arguments.
Before diving into the technical details, we offer some informal discussion of our approach and the intuition behind it.
A priori, it may be surprising that a block of type $\omega$ can ever pose a threat to the uniform relative $\Delta_3$ categoricity of the ordering.
After all, $\omega$ itself is relatively $\Delta_2$ categorical with automorphism orbits given by simply counting the number of predecessors that each element has.
It is just this that allows the orbits within $\omega$ to be sometimes defined even when placed inside the context of a larger ordering.
Consider the following example suggested by Proposition \ref{prop:Sigma3blockabove}.
Say that there is an $\omega$ block within an $sp$-homogeneous ordering of the form $\cdots+Sh(S)+\omega+Sh(S')+\cdots$ where $S$ and $S'$ are finite subsets of $\omega$.
By Proposition \ref{prop:Sigma3blockabove}, we can restrict our attention to $Sh(S)+\omega+Sh(S')$ using uniformly computable $\Sigma_3$ formulas.
Within this structure, we are then permitted to use a slight variation of the ordinary definition of the orbits within $\omega$ to obtain definitions of the orbits in $Sh(S)+\omega+Sh(S')$ (as in Proposition \ref{prop:Delta3begin}).
The takeaway here is that sometimes the local behavior of an \textit{sp}-homogeneous function near the $\omega$ block helps us define the orbits within it.
In contrast, some local behavior stops us from having nice orbit definitions.
An example that we explore more precisely in the formal proposition below is when the ordering is of the form $\cdots+Sh(T)+\omega+Sh(T')+\cdots$ where $T$ and $T'$ are infinite subsets of $\omega$.
In this case, one could still apply Proposition \ref{prop:Sigma3blockabove} as we did previously to isolate $Sh(T)+\omega+Sh(T')$ with $\Sigma_3$ formulas.
That said, this is not as helpful a move here, because we are unable to apply Proposition \ref{prop:Delta3begin} to obtain actual definitions from this point.
In fact, an argument using freeness similar to the one seen in Proposition \ref{prop:forbiddenShuffles} shows that this arrangement forbids the existence of $\Sigma_3$ definitions for orbits.
Intuitively, the $\omega$ block is too mixed in with arbitrarily large finite blocks to extract an appropriate definition. 
Our approach in the following proposition, then, is to enumerate all of the possible local behaviors near $\omega$ (as understood through the lens of the cases in Proposition \ref{prop:CasesDef}) and determine which local behavior gives rise to $\Sigma_3$ definitions for orbits and which ones forbid such definitions.
By ultimately exhausting all of the different possibilities, we give a full accounting of the allowable and disallowable behavior of an $\omega$ block within a uniformly relatively $\Delta_3$ categorical $sp$-homogeneous linear ordering.

\begin{proposition}\label{prop:omegaAnalysis}
    Consider an \textit{sp}-homogeneous linear ordering $\L=\A+\omega+\B$ where the written $\omega$ block is the only $\omega$ block (i.e., $\omega$ does not appear in a shuffle sum).
    \begin{enumerate}
        \item If $\A$ is in Case 3, then $\L$ is not uniformly relatively $\Delta_3$ categorical.
        \item If $\A$ is in Case 1 so $\A=\A'+Sh(S)$ where $S\subseteq\omega$ is infinite, then $\L$ is not uniformly relatively $\Delta_3$ categorical.
        \item If $\A$ is in Case 1 so $\A=\A'+Sh(S)$ where $S\subseteq\omega$ is finite, then the tuples in $\omega$ have uniformly computable $\Sigma_3$ definable automorphism orbits.
        \item If $\A$ is in Case 1 so that $\A=\A'+Sh(S)$ where $S$ includes an infinite block and finitely many finite blocks:
        \begin{enumerate}
            \item If $\B$ is in Case 1* with $\B=Sh(T)+\B'$ and $T$ has a non-zero, finite number of finite blocks, then the tuples in $\omega$ have uniformly computable $\Sigma_3$ definable automorphism orbits.
            \item If $\B$ is in Case 1* with $\B=Sh(T)+\B'$ and $T\subseteq \omega$ is infinite then $\L$ is not uniformly relatively $\Delta_3$ categorical.
            \item If $\B$ is in Case 2* with $\B=k+\B'$ for some $k\in\omega$ then the tuples in $\omega$ have uniformly computable $\Sigma_3$ definable automorphism orbits.
            \item If $\B$ is in Case 3* then $\L$ is not uniformly relatively $\Delta_3$ categorical.
        \end{enumerate}
    \end{enumerate}
\end{proposition}

\begin{proof}
    We begin with (1).
    More explicitly, $\L=\A+\omega+\B$ where $\A$ has arbitrarily large finite blocks in each final segment.
    Fix an element $a$ lying in $\omega$.
    We claim that $a$ is $3$-free.
    Let $\bar{b}$ be the tuple selected by the Spoiler.
    Write $\bar{b}=\bar{b}_1\bar{b}_2\bar{b}_3$ where $\bar{b}_1$ contains elements in $\A$, $\bar{b}_2$ contains elements in $\omega$ and $\bar{b}_3$ contains elements in $\B$.
    Without loss of generality, $\bar{b}_2$ is a finite interval of size $m$ including $a$.
    Say that $a$ is the $k^{th}$ element of $\bar{b}_2$.
    Let $a'$ be the $k^{th}$ element of a finite block, $F$ greater than any element finitely far from the elements in $\bar{b}_1$ and still in $\A$ (this is possible because $\A$ is in Case 3).
    Let $\bar{b}'=\bar{b}_1\bar{b}'_2\bar{b}_3$ where $\bar{b}'_2$ is the first $m$ elements of $F$.
    It is immediate that $\bar{a}\not\geq_3\bar{a}'$ as the two elements have a different number of successors. 
    We claim that $\bar{a}\bar{b}\leq_2\bar{a}'\bar{b}'$.
    Checking interval by interval, the only non-isomorphic corresponding intervals are the ones below and above $\bar{b}_2$ and $\bar{b}'_2$.
    Write the interval above $\bar{b}_2$ as $\omega+\C+K$ where $K$ is the (potentially empty) last block and $\C$ has no last element.
    The interval above $\bar{b}'_2$ is then isomorphic to $l+P+\omega+\C+K$ for some $l\in\omega$.
    The fact that
    \[l+P+\omega+\C+K\geq_2 \omega+\C+K\]
    follows from Theorem \ref{thm:2bnf}.
    The interval below $\bar{b}_2$ is $K+\D+n$ where $K$ is the (potentially empty) first block, $\D$ is in Case 3 with no least element, and $n\in\omega$.
    The interval below $\bar{b}'_2$ is then isomorphic to $K+\D'$, where $\D'$ has no last or least element and $\D'\hookrightarrow\D$ initially.
    We now observe that
    \[K+\D'\geq_2K+\D+n.\]
    In particular, if we let $\iota: K+\D'\to K+\D+n$ be the natural initial embedding, playing along $\iota$ is an easily confirmed winning strategy for Duplicator.

    (2) follows from a similar argument to the one above.
    In this case, $\L=\A'+Sh(S)+\omega+\B$ for some infinite $S\subseteq\omega$.
    It is just different enough to be included for completeness.
    Fix an element $a$ lying in $\omega$,.
    We claim that $a$ is $3$-free.
    Let $\bar{b}$ be the tuple selected by the Spoiler.
    Write $\bar{b}=\bar{b}_1\bar{b}_2\bar{b}_3$ where $\bar{b}_1$ contains elements in $\A$ and (without loss of generality) contains a largest block in $Sh(S)$, $\bar{b}_2$ contains elements in $\omega$ and $\bar{b}_3$ contains elements in $\B$.
    Without loss of generality, $\bar{b}_2$ is a finite interval of size $m$ including $a$.
    Say that $a$ is the $k^{th}$ element of $\bar{b}_2$.
    Let $a'$ be the $k^{th}$ element of a finite block, $F$ greater than any element finitely far from the elements in $\bar{b}_1$ in $Sh(S)$ (this is possible because $S$ is infinite).
    Let $\bar{b}'=\bar{b}_1\bar{b}'_2\bar{b}_3$ where $\bar{b}'_2$ is the first $m$ elements of $F$.
    It is immediate that $\bar{a}\not\geq_3\bar{a}'$ as the two elements have a different number of successors. 
    We claim that $\bar{a}\bar{b}\leq_2\bar{a}'\bar{b}'$.
    Checking interval by interval, the only non-isomorphic corresponding intervals are the ones above and below $\bar{b}_2$ and $\bar{b}'_2$.
    Write the interval above $\bar{b}_2$ as $\omega+\C+K$ where $K$ is the (potentially empty) last block and $\C$ has no last element.
    The interval above $\bar{b}'_2$ is then isomorphic to $l+Sh(S)+\omega+\C+K$ for some $l\in\omega$.
    The fact that
    \[l+Sh(S)+\omega+\C+K\geq_2 \omega+\C+K\]
    follows from Theorem \ref{thm:2bnf}.
    The interval below $\bar{b}_2$ is $Sh(S)+n$ where $n\in\omega$
    The interval below $\bar{b}'_2$ is isomorphic to $Sh(S)$.
    The fact that
    \[Sh(S)\geq_2Sh(S)+n\]
    follows from Theorem \ref{thm:2bnf}.

    We now move to (3).
    In this case, $\L=\A'+Sh(S)+\omega+\B$ where $S$ is a finite subset of $\omega$.
    We define the following helper predicates.
    Given a tuple of $k$ elements $\bar{y}$, we let $S_k(\bar{y})$ be the $\Pi_1$ statement that the elements of $\bar{y}$ are in a successor chain.
    Given a tuple of $k$ elements $\bar{y}$ we let $Bk_k(\bar{y})$ be the $\Pi_2$ statement that the elements of $\bar{y}$ form a size $k$ block.
    Given $N=\max(S)$, define the $\Sigma_3$ formula
    \[\chi_{n,S}(x):=S_N(x)\land \lnot P_{n+1}(x) \land \exists b<y_1<\cdots<y_n<x ~ \Big(S_{n+1}(\bar{y},x)\land  \]
    \[\forall z_0<\cdots<z_N\in(b,y_1)\lnot S_{N+1}(\bar{z})\land\bigwedge_{t\in S}\exists z_1<\cdots<z_t\in(b,y_1)~Bk_t(\bar{z})\Big).\]
    In words, $\chi_{n,T}(x)$ holds if $x$ 
    \begin{itemize}
        \item has $N$ successors
        \item does not have $n+1$ predecessors
        \item has $n$ predecessors
        \item below the $n^{th}$ predecessor and above some element $b$, there is no sequence of $N+1$ successors
        \item below the $n^{th}$ predecessor and above that same element $b$, there are blocks of size $t$ for every $t\in S$.
    \end{itemize}

    We claim that $\chi_{n,S}(x)$ holds exactly of the element in position $n$ of the written $\omega$.
    Consider an element $x_0\in\L$ such that $\L\models\chi_{n,S}(x)$.
    Consider the possibility that $x_0\in\A'$.
    Note that $x_0$ does not have blocks of size $t\in S$ below it by $sp$-homogeneity, so $\chi_{n,S}(x)$ cannot hold of such an element by the last bullet point.
    Consider the possibility that $x_0\in Sh(S)$. Note that $x_0$ does not have $N$ successors, so $\chi_{n,S}(x)$ cannot hold of such an element by the first bullet point.
    Consider the possibility that $x_0\in \B$.
    There must be a witness $b_0$ to statements in the fourth and fifth bullet points.
    $b_0$ cannot be in $\B$ by the fifth bullet point, as then there are no blocks of size $t\in S$ above $b_0$ by $sp$-homogeneity.
    We have that $b_0$ cannot be below $\B$ by the fourth bullet point, as then there would be an arbitrarily long successor chain in the form of the $\omega$ block between $b_0$ and $x_0$.
    Therefore, $x_0$ cannot be in $\B$.
    Any element that is not in position $n$ of the $\omega$ either has $n+1$ predecessors or does not have $n$ predecessors, so $\chi_n(x)$ cannot hold of such an element by the second and third bullet points. 
    Lastly, it is direct to confirm that $\chi_n(x)$ holds of the element in position $n$ of the $\omega$; the witness for $b$ can be any element from $Sh(S)$.
    Therefore, $\chi_{n,S}(x)$ isolates the automorphism orbit of this element.
    Taking conjunctions over the isolating formulas defined above yields the desired automorphism definitions for tuples in the $\omega.$

    We now analyze (4).
    In (a) and (c), we have that $\L=A'+Sh(S)+\omega+Sh(T)+\B'$ or $\L=A'+Sh(S)+\omega+k+\B'$ where $S$ includes finitely many finite blocks and an infinite block and $T$ is a finite set of finite blocks. 
    Observe that by Corollary \ref{cor:defBoundedSummand} it is enough to define the orbits in $\omega$ within $Sh(S)+\omega+Sh(T)$ and $Sh(S)+\omega+k$ respectively.
    In both cases, we can let $N=\max(S,T,k)$ and define the $n^{th}$ element of $\omega$ by saying
    \[\lnot P_{n+1}(x)\land S_{N+1}(x)\land P_n(x).\]
    By assumption, $S_{N+1}(x)$ can only hold of elements in infinite blocks in the ordering.
    Any element with $\lnot P_{n+1}(x)$ must be in the $\omega$ block as it cannot be in a $\zeta$ or $\omega^*$ block.
    Within the $\omega$ block $\lnot P_{n+1}(x)\land P_n(x)$ clearly gives the desired definition.

    We can therefore move to the analysis of (b) and (d).
    In these cases, we have that $\L=A'+Sh(S)+\omega+Sh(T)+\B'$ or $\L=A'+Sh(S)+\omega+\B$ where $S$ includes finitely many finite blocks and an infinite block, $T$ is a finite set of finite blocks, and $\B$ contains arbitrarily large finite blocks in each initial segment.
    These arguments are similar to those in (1) and (2).
    We provide the proof of just (d), as it is the more involved of the two, and the proof of (b) is quite similar to what has already been seen.
    Fix an element $a$ lying in $\omega$,.
    We claim that $a$ is $3$-free.
    Let $\bar{b}$ be the tuple selected by the Spoiler.
    Write $\bar{b}=\bar{b}_1\bar{b}_2\bar{b}_3$ where $\bar{b}_1$ contains elements in $\A$, $\bar{b}_2$ contains elements in $\omega$ and $\bar{b}_3$ contains elements in $\B$.
    Without loss of generality, $\bar{b}_2$ is a finite interval of size $m$ including $a$.
    Say that $a$ is the $k^{th}$ element of $\bar{b}_2$.
    Let $a'$ be the $k^{th}$ element of a finite block, $F$, less than any element finitely far from the elements in $\bar{b}_3$ and in $\B$ (this is possible because $\B$ is in Case 3*).
    Let $\bar{b}'=\bar{b}_1\bar{b}'_2\bar{b}_3$ where $\bar{b}'_2$ is the first $m$ elements of $F$.
    It is immediate that $\bar{a}\not\geq_3\bar{a}'$ as the two elements have a different number of successors. 
    We claim that $\bar{a}\bar{b}\leq_2\bar{a}'\bar{b}'$.
    Checking interval by interval, the only non-isomorphic corresponding intervals are the ones below and above $\bar{b}_2$ and $\bar{b}'_2$.
    Write the interval above $\bar{b}_2$ as $\omega+\C+K$ where $K$ is the (potentially empty) last block and $\C$ has no last element.
    The interval above $\bar{b}'_2$ is then isomorphic to $l+\C'+K$ for some $l\in\omega$ and $\C'$ with no last element.
    The fact that
    \[l+\C'+K\geq_2 \omega+\C+K\]
    follows from Theorem \ref{thm:2bnf}.
    The interval below $\bar{b}_2$ is $K+\D+n$ where $K$ is the (potentially empty) first block, $\D$ is in Case 3 with no least element, and $n\in\omega$.
    The interval below $\bar{b}'_2$ is then isomorphic to $K+\D'$, where $\D'$ has no last or least element and $k\in\omega$.
    The fact that
    \[K+\D'+k\geq_2K+\D+n\]
    follows from Theorem \ref{thm:2bnf}.
\end{proof}

We summarize the results from the above proposition in the table below
The notation $Sh(inf, bdd)$ refers to a shuffle sum of finitely many finite blocks and an infinite block.
\begin{center}
\begin{tabular}{c|c|c|c|c}
     $Col+\omega+row$& Case 3 & $\A'+Sh(ubd)$ & $\A'+Sh(inf, bdd)$ & $\S'+Sh(bdd)$\\
     \hline
   Case 3  & $\Delta_4$ & $\Delta_4$ & $\Delta_4$ & $\Delta_3$  \\   \hline
   $Sh(ubd)+\B'$  & $\Delta_4$ & $\Delta_4$ & $\Delta_4$ & $\Delta_3$ \\  \hline
   $Sh(inf, bdd)+\B'$  & $\Delta_4$ & $\Delta_4$ & $\Delta_3$ & $\Delta_3$ \\  \hline
   $Sh(bdd)+\B'$  & $\Delta_4$ & $\Delta_4$ & $\Delta_3$ & $\Delta_3$ \\  \hline
   $k+\B'$  & $\Delta_4$ & $\Delta_4$ & $\Delta_3$ & $\Delta_3$
\end{tabular}
\end{center}

This table uses the column labels to indicate the linear ordering below $\omega$ and the row labels to indicate the linear ordering above $\omega$.
$\Delta_4$ is used to indicate that this combination explicitly precludes the case that the overall linear ordering is $\Delta_3$ uniformly relatively categorical.
$\Delta_3$ is used to indicate that the automorphism orbits in the copy of $\omega$ are appropriately definable, in other words, that the elements of $\omega$ pose no challenge to the linear orderings being $\Delta_3$ uniformly relatively categorical.

Conspicuously missing from the above results are the cases where $\omega$ is abutted by other purely infinite blocks (i.e., $\zeta, \omega^*, \zeta\cdot\eta, \omega^*\cdot\eta$).
These cases have not been forgotten; they will be addressed later on.
We first note the following generalization of certain parts of the above theorem.

\begin{theorem}\label{thm:generalizedFree}
    Let $I$ contain only infinite blocks. Let $\A+I+\B$ be $sp$-homogeneous.
    \begin{enumerate}
        \item If $\A$ is in Case 3 and $I$ has no first element then $\L$ is not uniformly relatively $\Delta_3$ categorical.
        \item If $\B$ is in Case 3* and $I$ has no last element then $\L$ is not uniformly relatively $\Delta_3$ categorical.
        \item If $\A$ is in Case 1 so $\A=\A'+Sh(S)$ where $S\subseteq\omega$ is infinite and $I$ has no first element then $\L$ is not uniformly relatively $\Delta_3$ categorical.
        \item If $\B$ is in Case 1* so $\B=Sh(S)+\B'$ where $S\subseteq\omega$ is infinite and $I$ has no last element then $\L$ is not uniformly relatively $\Delta_3$ categorical.
    \end{enumerate}
\end{theorem}

\begin{proof}
   The same proofs or even simpler proofs (or the same approach applied to $\L^*$) as those seen in Proposition \ref{prop:omegaAnalysis} (1) and (2) also demonstrate these results.
   For this reason, we do not provide the proofs of these facts in complete detail, instead preferring to highlight a few key cases.

    We will exclusively refer to the second and fourth of the above claims, noting that the first and the third follow by symmetry.
    Note that there are only finitely many possible infinite block types, so $I$ can be written as a finite sum of the pure infinite blocks not disallowed by Proposition \ref{prop:forbiddenShuffles}: $\zeta, \omega^*, \zeta\cdot\eta, \omega^*\cdot\eta$, $\omega$, and $Sh(\omega,\omega^*)$.
    In particular, $I$ has a final, convex subordering $J$ of one of the above types.
    Assume that $I$ has no greatest element; $J$ must also have no greatest element.
    If that ordering is of type $\omega$ or type $\zeta$, the same \textit {exact} argument from Proposition \ref{prop:omegaAnalysis} (1) and (2) shows that the first and third claims hold.

    There is a slight variation of the argument in the case that $J$ is a dense arrangement of infinite blocks.
    For concreteness, say that $J=\omega\cdot\eta$; all of the other cases are precisely the same.
    We show that any initial element $a$ of one of the $\omega$ blocks is 3-free when $\B$ is in Case 3.
    This shows the second of the above claims, while the fourth of the above claims follows from a similar (simpler) argument.
    Let $\bar{b}$ be the tuple selected by the Spoiler.
    Write $\bar{b}=\bar{b}_1\bar{b}_2\bar{b}_3\bar{b}_4$ where $\bar{b}_1$ contains elements below $a$, $\bar{b}_2$ contains elements in the copy of $\omega$ containing $a$, $\bar{b}_3$ contains the elements in a greater $\omega$ block than $a$ within $\omega\cdot\eta$, and $\bar{b}_4$ contains elements in $\B$.
    Without loss of generality, $\bar{b}_2$ is a finite interval of size $m$ including $a$.
    Similarly, $\bar{b}_3=\bar{b}_3^1\cdots\bar{b}_3^k$ is a finite sequence of finite intervals of size $m$ containing the first element of each $\omega$ they intersect.
    Let $a',b_3^{1'}\cdots b_3^{k'}$ be the first elements of finite blocks, $F_i$, less than any element finitely far from the elements in $\bar{b}_3$ and in $\B$ (this is possible because $\B$ is in Case 3*).
    Let $\bar{b}'=\bar{b}_1\bar{b}'_2\bar{b}'_3\bar{b}_4$ where $\bar{b}'_2$ is the first $m$ elements of $F_0$ and $\bar{b}'_3$ is the first $m$ elements of each of the $F_i$ with $i>0$.
    It is immediate that $\bar{a}\not\geq_3\bar{a}'$ as the two elements have a different number of successors. 
    We claim that $\bar{a}\bar{b}\leq_2\bar{a}'\bar{b}'$.
    Checking interval by interval, the only non-isomorphic corresponding intervals are the ones below, above, and between $\bar{b}_2,\bar{b}_3$ and $\bar{b}'_2,\bar{b}_3'$.
    The intervals between $\bar{b}_2,\bar{b}_3$ are all isomorphic to $\omega+\omega\cdot\eta$.
    The intervals between $\bar{b}_2',\bar{b}_3'$ are all isomorphic to some $\P$, which is infinite and has no least element, by construction.
    The fact that
    \[\P\geq_2 \omega+\omega\cdot\eta\]
    follows from Theorem \ref{thm:2bnf}.
    The interval below $\bar{b}_2$ is $K+\D+Sh(\omega)$ where $K$ is the (potentially empty) first block, $\D$ is some arbitrary linear ordering with no least element.
    The interval below $\bar{b}'_2$ is then isomorphic to $K+\D+Sh(\omega)+\E$ where $\E$ has no greatest element.
    The fact that
    \[K+\D+Sh(\omega)+\E \geq_2K+\D+Sh(\omega)\]
    follows from Theorem \ref{thm:2bnf}.
    The interval above $\bar{b}_3$ is $Sh(\omega)+\D$ where $\D$ is some arbitrary linear ordering.
    The interval above $\bar{b}'_3$ is then isomorphic to $\D'$ where $\D'$ is an infinite, final segment of $\D$.
    The fact that
    \[\D' \geq_2Sh(\omega)+\D\]
    follows from Theorem \ref{thm:2bnf}.

    With all other arguments being analogous, the theorem follows.
\end{proof}

We can use these observations to understand the behavior of many other block types.

\begin{corollary}
    Consider an $sp$-homogeneous linear ordering $\L=\A+I+\B$ where  $I=\zeta$, $\omega\cdot\eta$, $Sh(\omega,\omega^*)$, $\omega^*\cdot\eta$ or $\zeta\cdot\eta$.
    \begin{enumerate}
        \item If $\A$ or $\B$ is in Case 3, then $\L$ is not uniformly relatively $\Delta_3$ categorical.
        \item If $\A$ is in Case 1 so $\A=\A'+Sh(S)$ where $S\subseteq\omega$ is infinite then $\L$ is not uniformly relatively $\Delta_3$ categorical.
        \item If $\B$ is in Case 1* so $\B=\B'+Sh(S)$ where $S\subseteq\omega$ is infinite then $\L$ is not uniformly relatively $\Delta_3$ categorical.
        \item If $\A=\A'+k$ or $\A=\A'+Sh(S)$ with $S\subseteq\omega\cup\{\zeta,\omega,\omega^*\}$ finite and $\B=l+\B'$ or $\B=Sh(T)+\B'$ with $T\subseteq\omega$, then the tuples in $I$ have uniformly computable $\Sigma_3$ definable automorphism orbits.
    \end{enumerate}   
    
\end{corollary}

\begin{proof}
    The first three points follow immediately from Theorem \ref{thm:generalizedFree}.
    The last bullet point follows from Corollary \ref{cor:defBoundedSummand} and Proposition \ref{prop:Delta3begin}.
\end{proof}

It is worth noting explicitly that in this case, where there are no first or last elements of our purely infinite blocks, the only cases where the ordering is uniformly relatively $\Delta_3$ categorical are when $I$ is surrounded by "buffers" with finitely many finite blocks and we can approach the definition of the automorphism orbits using Corollary \ref{cor:defBoundedSummand} and Proposition \ref{prop:Delta3begin}.
More specifically, the orbits are always of the form $\varphi_{n_1,<}(x)\land \varphi_{n_2,>}\land \psi$ where $\psi,$ as in Proposition \ref{prop:Delta3begin}, describes a number of predecessors and successors.

The following table summarizes our results for $I=\zeta$.
It also applies to $I=\omega\cdot\eta$, $Sh(\omega,\omega^*)$, $\omega^*\cdot\eta$ or $\zeta\cdot\eta$, or in fact, and $I$ with no first or last element.

\begin{center}
\begin{tabular}{c|c|c|c|c}
     $Col+\zeta+row$& Case 3 & $\A'+Sh(ubd)$ & $\A'+Sh(bdd)$ & $\S'+k$\\
     \hline
   Case 3  & $\Delta_4$ & $\Delta_4$ & $\Delta_4$ & $\Delta_4$  \\   \hline
   $Sh(ubd)+\B'$  & $\Delta_4$ & $\Delta_4$ & $\Delta_4$ & $\Delta_4$ \\  \hline
   $Sh(bdd)+\B'$  & $\Delta_4$ & $\Delta_4$ & $\Delta_3$ & $\Delta_3$ 
   \\  \hline
   $k+\B'$  & $\Delta_4$ & $\Delta_4$ & $\Delta_3$ & $\Delta_3$
\end{tabular}
\end{center}

Note that the behavior of $\omega^*$ is ascertained by reversing all of the orderings in the claims and conclusions for $\omega$.
This means that we only need to understand the scenarios where multiple infinite blocks appear adjacent to each other to complete our analysis.

\begin{proposition}
    Let $I$ be a sum of $\omega$, $\omega^*$, $\zeta$, $\omega\cdot\eta$, $Sh(\omega,\omega^*)$, $\omega^*\cdot\eta$, or $\zeta\cdot\eta$ with at least two terms and no repeating infinite block types.
    Let $\L=\A+I+\B$.
        \begin{enumerate}
        \item If $\A$ or $\B$ is in Case 3, then $\L$ is not uniformly relatively $\Delta_3$ categorical.
        \item If $\A$ is in Case 1 so $\A=\A'+Sh(S)$ where $S\subseteq\omega$ is infinite then $\L$ is not uniformly relatively $\Delta_3$ categorical.
        \item If $\B$ is in Case 1* so $\B=\B'+Sh(S)$ where $S\subseteq\omega$ is infinite then $\L$ is not uniformly relatively $\Delta_3$ categorical.
        \item If $\A=\A'+P$ with $P=Sh(S)$ or $k$ with $S\subseteq\omega\cup\{\zeta, \omega, \omega^*\}$ finite and $\B=Q+\B'$ with $Q=Sh(T)$ or $l$ with $T\subseteq\omega\cup\{\zeta, \omega, \omega^*\}$ finite, then the tuples in $I$ have uniformly computable $\Sigma_3$ definable automorphism orbits if and only if $P+I+Q$ is uniformly relatively $\Delta_3$ categorical.
    \end{enumerate} 
\end{proposition}

    \begin{proof}
        Say that $\A$ is in Case 3 or that it is in Case 1 where $S\subseteq\omega$ is infinite.
        By Theorem \ref{thm:generalizedFree}, the last summand of $I$ must be $\omega^*$.
        This means that the next summands $K$ of $I$ must have no last element. 
        Using Theorem \ref{thm:generalizedFree} on $\A+K$ gives that $\L$ is not uniformly relatively $\Delta_3$ categorical.
        A symmetrical proof works for when $\B$ is in Case 3* or in Case 1* with $S\subseteq\omega$ infinite.
        
        This leaves only the possibility that we are in (4) as above.
        The result follows immediately from Corollary \ref{cor:defBoundedSummand}.
     \end{proof}

    What remains is understanding the categoricity of linear orderings of the form $P+I+Q$ as in the fourth case of the above proposition.

    \begin{proposition}\label{prop:infiniteAnalysis}
        \begin{enumerate}
            \item $\zeta+Sh(S)$ with $S\subseteq\omega\cup\{\omega,\omega^*\}$ containing at least one infinite block is not uniformly relatively $\Delta_3$ categorical.
            \item $\zeta\cdot\eta+Sh(S)$ with $S\subseteq\omega\cup\{\omega,\omega^*\}$ containing at least one infinite block is not uniformly relatively $\Delta_3$ categorical.
            \item Let $\L=P+I+Q$, where $P=Sh(S)$ or $k$ with $S\subseteq\omega\cup\{\zeta,\omega,\omega^*\}$ finite, $Q=Sh(T)$ or $l$ with $T\subseteq\omega\cup\{\zeta,\omega,\omega^*\}$ finite and $I$ containing only infinite blocks. As long as a block isomorphic to $\zeta$ or $\zeta\cdot\eta$ is not adjacent to a block isomorphic to $Sh(S)$ with $S\subseteq\omega\cup\{\omega,\omega^*\}$, it is uniformly relatively $\Delta_3$ categorical.
        \end{enumerate}
    \end{proposition}

    \begin{proof}
    We begin with (1).
    Fix an element $a$ lying in $\zeta$.
    We claim that $a$ is $3$-free.
    Let $\bar{b}$ be the tuple selected by the Spoiler.
    Write $\bar{b}=\bar{b}_1\bar{b}_2$ where $\bar{b}_1$ contains elements in $\zeta$ and $\bar{b}_2$ contains elements in $Sh(S)$.
    Without loss of generality, $\bar{b}_1$ is a finite interval of size $m$ including $a$.
    Say that $a$ is the $k^{th}$ element of $\bar{b}_1$.
    Let $a'$ have at least $m$ successors and predecessors and be an element of an infinite block, $J$, in $Sh(S)$ below all elements of $\bar{b}_2$.
    Let $\bar{b}'=\bar{b}'_1\bar{b}_2$ where $\bar{b}'_1$ is a block of size $m$ with $a'$ in the $k^{th}$ position within $J$.
    It is immediate that $\bar{a}\not\geq_3\bar{a}'$ as the two elements have a different number of successors or predecessors (which one depends on whether $J$ is an $\omega$ or $\omega^*$ block). 
    We claim that $\bar{a}\bar{b}\leq_2\bar{a}'\bar{b}'$.
    Checking interval by interval, the only non-isomorphic corresponding intervals are the ones above and below $\bar{b}_1$ and $\bar{b}'_1$.
    Write the interval above $\bar{b}_1$ as $\omega+Sh(S)$.
    The interval above $\bar{b}'_1$ is then isomorphic to $K+Sh(S)$ where $K$ is some (potentially empty) block with a first element.
    The fact that
    \[K+Sh(S) \geq_2 \omega+Sh(S)\]
    follows from Theorem \ref{thm:2bnf}.
    The interval below $\bar{b}_1$ is isomorphic to $\omega^*$.
    The interval below $\bar{b}'_1$ is isomorphic to $\zeta+Sh(S)+K$, where $K$ is some (potentially empty) block with a last element.
    The fact that
    \[\zeta+Sh(S)+K \geq_2\omega^*\]
    follows from Theorem \ref{thm:2bnf}.

    We now check (2), which is quite similar to (1).
    Fix an element $a$ lying in $\zeta\cdot\eta$.
    We claim that $a$ is $3$-free.
    Let $\bar{b}$ be the tuple selected by the Spoiler.
    Write $\bar{b}=\bar{b}_1\bar{b}_2$ where $\bar{b}_1$ contains elements in $\zeta\cdot\eta$ and $\bar{b}_2$ contains elements in $Sh(S)$.
    Without loss of generality, $\bar{b}_1$ is a finite sequence of finite intervals of size $m$ including $a$.
    Say that $a$ is the $k^{th}$ element of its interval within $\bar{b}_1$.
    Let $a'$ have at least $m$ successors and predecessors and be an element of an infinite block, $J$, in $Sh(S)$ below all of the elements of $\bar{b}_2$.
    Let $\bar{b}'=\bar{b}'_1\bar{b}_2$ where $\bar{b}'_1$ is a sequence of blocks of size $m$ where the block corresponding to the block with $a$ in it has $a'$ in it in the $k^{th}$ position within $J$.
    It is immediate that $\bar{a}\not\geq_3\bar{a}'$ as the two elements have a different number of successors or predecessors (which one depends on whether $J$ is an $\omega$ or $\omega^*$ block). 
    We claim that $\bar{a}\bar{b}\leq_2\bar{a}'\bar{b}'$.
    Checking interval by interval, the only non-isomorphic corresponding intervals are the ones between, above, and below the blocks of $\bar{b}_1$ and $\bar{b}'_1$.
    Write the interval above $\bar{b}_1$ as $\omega+\zeta\cdot\eta+Sh(S)$.
    The interval above $\bar{b}'_1$ is then isomorphic to $K+Sh(S)$ where $K$ is some (potentially empty) block with a first element.
    The fact that
    \[K+Sh(S) \geq_2 \omega+\zeta\cdot\eta+Sh(S)\]
    follows from Theorem \ref{thm:2bnf}.
    The interval below $\bar{b}_1$ is isomorphic to $\zeta\cdot\eta+\omega^*$.
    The interval below $\bar{b}'_1$ is isomorphic to $\zeta\cdot\eta+Sh(S)+K$ where $K$ is some (potentially empty) block with a last element.
    The fact that
    \[\zeta\cdot\eta+Sh(S)+K \geq_2\zeta\cdot\eta+\omega^*\]
    follows from Theorem \ref{thm:2bnf}.
    The intervals in between the blocks $\bar{b}_1$ are isomorphic to $\omega+\zeta\cdot\eta+\omega^*$.
    The intervals in between the blocks $\bar{b}'_1$ are isomorphic to $K+Sh(S)+K'$, where $K$ is some (potentially empty) block with a first element and $K'$ is some (potentially empty) block with a least element.
    The fact that 
    \[K+Sh(S)+K' \geq_2\omega+\zeta\cdot\eta+\omega^*\]
    follows from Theorem \ref{thm:2bnf}.

    We lastly analyze (3).
    Note that a linear ordering $\L=P+I+Q$ only has bounded finite block sizes.
    Therefore, by Proposition \ref{prop:Delta3begin}, there must be a $\zeta$ block and an $\omega$ or $\omega^*$ block in $\L$ for there to be an issue with uniform relative $\Delta_3$ categoricity.
    Furthermore, the definition of the orbits for the elements in $\omega$ given in  Proposition \ref{prop:Delta3begin}, $P_k(x)\land S_N(x)\land \lnot P_{k+1}(x)$, and their dual definitions for the elements in $\omega^*$, $S_k(x)\land P_N(x)\land \lnot S_{k+1}(x)$, still define their orbits in $\L$.
    In other words, the only potential issue is that the orbits in the $\zeta$ block(s) are not definable in the desired manner.
    By assumption, we are in the case where $Z=\zeta$ or $\zeta\cdot\eta$ locally appear with $\A+Z+\B$ and $\A,\B$ are equal to $n,\omega,\omega^*$, or $Sh(S)$ with $S\subseteq \omega$ finite.
    We claim that, given this, $Z$ is computably $\Sigma_3$ definable, giving the desired result by relativizing the definitions from Proposition \ref{prop:Delta3begin} to this definition.
    In the case that $\A$ has only finite blocks, the elements to the right of $\A$ (including $\A$ itself) can be defined as being to the right of blocks of size $\A$ as in Corollary \ref{cor:defBoundedSummand}.
    In the case that $\A$ is a single infinite block, elements to the right of $\A$ (not including $\A$ itself) are infinitely to the right of an element satisfying the defining formula of an automorphism orbit in $\A$ (this can be said in a $\Sigma_3$ manner).
    Proceeding dually with $\B$ yields a computable $\Sigma_3$ definition for $\A+Z+\B$ where $\A$ and $\B$ are either empty or only have finite blocks.
    The elements in $Z$ are now simply defined by $S_N(x)$ for sufficiently large $N$.
    \end{proof}

    Putting this all together yields the following theorem.

    \begin{theorem} \label{spud3cat}
        An $sp$-homogeneous linear ordering $\L$ is uniformly relatively $\Delta_3$ categorical if and only if: 
        \begin{enumerate}
            \item It has no intervals of the form $Sh(S)$ where $S$ includes an infinite block and finite blocks of arbitrary size or $\zeta$ along with another finite block.
            \item If an interval $I$ is of the form $\zeta$, $\omega\cdot\eta$, $Sh(\omega,\omega^*)$, $\omega^*\cdot\eta$ or $\zeta\cdot\eta$ or a sum of at least two of $\omega$, $\omega^*$, $\zeta$, $\omega\cdot\eta$, $Sh(\omega,\omega^*)$, $\omega^*\cdot\eta$ or $\zeta\cdot\eta$, then $I$ has an interval to its left and right in $\L$ that only has finite blocks of a bounded size.
            \item Any single $\omega$ block has an interval to its left with only finite blocks of bounded size or an interval to its left and right with finite blocks of bounded size.
            \item Any single $\omega^*$ block has an interval to its right with only finite blocks of bounded size or an interval to its left and right with finite blocks of bounded size.
            \item Any $\zeta$ or $\zeta\cdot\eta$ does not have an interval to its right or left isomorphic to a shuffle sum including an infinite block.
        \end{enumerate}
    \end{theorem}

    \begin{proof}
        The results of this section, taken together, prove that the elements in infinite blocks of $\L$ with each of the five properties listed above have uniformly computable $\Sigma_3$ automorphism orbits without parameters.
        Combining these definitions with those for finite blocks provided in Proposition \ref{prop:Delta3begin} gives a uniformly computable $\Sigma_3$ Scott set for $\L$ without parameters.
        In particular, $\L$ is uniformly relatively $\Delta_3$ categorical as desired.

        The results of this section also provide arguments showing that in each of the above scenarios, $\L$ has a $3$-free tuple.
        In particular, $\L$ cannot have a $\Sigma_3$ Scott set without parameters, so $\L$ is not uniformly relatively $\Delta_3$ categorical, as desired.
    \end{proof}

This is a characterization of the uniformly relatively $\Delta_3$, $sp$-homogeneous linear orderings.
We use this to also characterize the relatively $\Delta_3$, $sp$-homogeneous linear orderings.

\begin{proposition}
    An $sp$-homogeneous linear ordering $\L$ is relatively $\Delta_3$ categorical if and only if it can be written $\L=L_0+1+L_1+1+\cdots+1+L_n$ where each $L_i$ is uniformly relatively $\Delta_3$ categorical.
\end{proposition}

\begin{proof}
    Say that $\L$ can be written in such a manner.
    It is then the case that $(\L,\bar{a})$ is uniformly relatively $\Delta_3$ categorical, where the $\bar{a}$ consists of the named "1"s in the above form.
    By Theorem VIII.4 of \cite{cst1}, this means that $\L$ is relatively $\Delta_3$ categorical.

    Say that $\L$ is relatively $\Delta_3$ categorical.
    By Theorem VIII.4 of \cite{cst1}, this means that for some $\bar{a}\in\L$,  $(\L,\bar{a})$ is uniformly relatively categorical.
    Let $\bar{a}=a_1<\cdots<a_k$ and let $a_0=-\infty$ and $a_k=\infty$.
    Note that $\L=(a_0,a_1)+\sum_{1\leq i\leq k}1+(a_i,a_{i+1})$.
    We claim that each $(a_i,a_{i+1})$ is uniformly relatively $\Delta_\alpha$ categorical.
    Say that some $(a_i,a_{i+1})$ was not.
    By the analysis in the above section, this means that $(a_i,a_{i+1})$ has some 3-free tuple.
    This tuple is still 3-free in $(\L,\bar{a})$ over $\bar{a}$.
    Therefore, $\L$ cannot be relatively $\Delta_3$ categorical.
\end{proof}

This yields the following.

 \begin{theorem} \label{spd3}
        A (weakly) $sp$-homogeneous linear ordering $\L$ is relatively $\Delta_3$ categorical if and only if: 
        \begin{enumerate}
            \item It has no intervals of the form $Sh(S)$ where $S$ includes an infinite block and finite blocks of arbitrary size or $\zeta$ along with another finite block.
            \item If $I$ is $\omega\cdot\eta$, $Sh(\omega,\omega^*)$, $\omega^*\cdot\eta$ or $\zeta\cdot\eta$ or a sum of at least two of those orderings, it has an interval to its left and right in $\L$ that only has finite blocks of a bounded size.
            \item Any $\zeta\cdot\eta$ does not have an interval to its right or left isomorphic to a shuffle sum including an infinite block.
        \end{enumerate}
    \end{theorem}

    \begin{proof}
        This is identical to the characterization for uniformly relatively $\Delta_3$ categorical except for the fact that there are no longer restrictions on the appearance of $\omega$, $\omega^*$, or $\zeta$ blocks.
        For $\omega$ blocks, one can take a parameter as the first element.
        In the resulting partition, the $\omega$ block is the first block, so it has an acceptable form.
        For $\omega^*$ blocks, the dual approach works.
        For $\zeta$ blocks, one can take a parameter in the middle of the block.
        In the resulting partition, the left side has a final $\omega^*$ and the right side an initial $\omega$.
        Both the initial $\omega$ (or final $\omega^*$) and its complement are definable by a computable $\Sigma_3$ formula (by counting the number of successors/predecessors), so it does not affect the categoricity of the resulting parts.

        This means that all we need to show is that the listed restrictions still apply even over parameters.
        For (1), every partition of such a $Sh(S)$ still has an interval with a $Sh(S)$ convexly embedding into it, so the resulting partition cannot be into uniformly relatively $\Delta_3$ categorical pieces.
        Similar logic shows that (2) holds.
        For example (all other cases are argued analogously), if $\omega\cdot\eta$ is to the right of an interval where there are arbitrarily large finite blocks to the left of every point, any partition still gives a part with $\omega\cdot\eta$ adjacent to an interval where there are arbitrarily large finite blocks to the left of every point.
        Therefore, this part cannot be uniformly relatively $\Delta_3$ categorical.
        Lastly, for (3), just as above, any partition of an ordering with an adjacent $\zeta\cdot\eta$ and shuffle sum including an infinite block still has this adjacency, so it cannot be uniformly relatively $\Delta_3$ categorical.

        This result works equally well for weakly $sp$-homogeneous linear orderings.
        This is achieved by adding the exceptional parameter set to the set of parameters in the case that $\L$ is relatively $\Delta_3$ categorical, and noting the existence of an unavoidable problematic interval, just as above, in the case that $\L$ is not relatively $\Delta_3$ categorical.
    \end{proof}

\subsection{Relative computable categoricity of $sp$-linear orderings}

In this section, we explore $sp$-homogeneous orderings, as structures of the form $(L,<,s,p)$.
Our goal is to classify when $sp$-homogeneous linear orderings are (uniformly) relatively computably categorical as $sp$-structures.
This is closely related to our endeavor in the previous section.
We see this clearly in the following lemma.

\begin{lemma}\label{lem:ConnectingSpTo<}
    Fix an $sp$-homogeneous linear ordering $\L$.
    If $(\L,<,s,p)$ is relatively computably categorical then $(\L,<)$ is relatively $\Delta_3^0$ categorical.
    If $(\L,<,s,p)$ is uniformly relatively computably categorical then $(\L,<)$ is uniformly relatively $\Delta_3^0$ categorical.
\end{lemma}

\begin{proof}
    Say that $(\L,<,s,p)$ is relatively computably categorical.
    This means that there is a $\Sigma_1$ Scott family in the language $<,s,p$ over some parameters $\bar{p}$.
    Let $\psi(\bar{x},\bar{p})$ be a formula in the Scott family. 
    As previously observed, the functions $s$ and $p$ are computable in two jumps of the structure $(\L,<)$.
    This means that the validity of any atomic formula of the form 
    \[s^{n_1}\circ p^{m_1}\circ\cdots\circ s^{n_r}\circ p^{m_r}(x)=s^{a_1}\circ p^{b_1}\circ\cdots\circ s^{a_k}\circ p^{b_k}(x)\]
    or
    \[s^{n_1}\circ p^{m_1}\circ\cdots\circ s^{n_r}\circ p^{m_r}(x)<s^{a_1}\circ p^{b_1}\circ\cdots\circ s^{a_k}\circ p^{b_k}(x)\]
    can also be calculated in two jumps of the structure $(\L,<)$.
    Thus, there is a $\Delta_3(<)$ way to express all atomic formulas that include $s$ or $p$.
    Define $\psi'(\bar{x},\bar{p})$ by replacing all instances of the above formulas with their equivalent definitions in $<$ using the $\Sigma_3$ definition when the atomic formula appears positively and the $\Pi_3$ definition when the atomic formula appears negatively.
    Because $\psi$ is $\Sigma_1$ this gives that $\psi'$ is $\Sigma_3$.
    Furthermore, by construction, $\psi'$ contains no instances of $p$ and $s$, and it is equivalent to $\psi$.
    This yields the desired $\Sigma_3$ Scott family for $(\L,<)$.

    To obtain the result for uniform computable categoricity, repeat the above argument without the use of parameters; it proceeds exactly the same way.
\end{proof}

We now analyze the extent to which the converse of the above lemma holds.
To be explicit, the converse does not always hold.
For the sake of illustration, consider the linear ordering $\L=\omega+Sh(\omega)$.
We know by Theorem \ref{spud3cat} that $\L$ is uniformly relatively $\Delta_3$ categorical as a pure linear ordering.
That said, it is apparently difficult to give a parameterless Scott family for this ordering.
Consider, for example, the first element.
The natural definition of its orbit would be the one that states it is the first element.
That said, this is a $\Pi_1$ formula, and it is unclear that adding definitions for $s$ and $p$ should allow for a $\Sigma_1$ alternative.
Indeed, our next lemma shows that there is no $\Sigma_1$ formula in this situation and related ones.

\begin{lemma}\label{lem:forbiddensp}
    Let $\L$ be an $sp-$homogeneous linear ordering where $\L=\A+\omega+\B$ and the written $\omega$ is the unique $\omega$ block in the ordering.
    If $\B$ has arbitrarily large finite blocks in every initial segment (i.e., it is in Case $3^*$ or begins with a shuffle sum containing infinitely many finite blocks), then $(\L,<,s,p)$ is not uniformly relatively computably categorical.

    Similarly, let $\L=\A+\omega^*+\B$, and the written $\omega^*$ is the unique $\omega$ block in the ordering.
    If $\A$ has arbitrarily large finite blocks in every final segment (i.e., it is in Case $3$ or ends with a shuffle sum containing infinitely many finite blocks), then $(\L,<,s,p)$ is not uniformly relatively computably categorical.
\end{lemma}

\begin{proof}
    We show that $x$, the initial element in the $\omega$ block, is 1-free.
    It is essential for this argument to recall that $\bar{x}\leq_0\bar{y}$ holds not in the case that $\bar{x}$ and $\bar{y}$ agree on \textit{every} quantifier-free formula, but rather the first $|\bar{y}|=|\bar{x}|$ many quantifier-free formulas.
    By convention, fix an enumeration of $sp-$quantifier-free formulas where functions are not composed $k+1$ times in any of the first $k$ formulas.
    Say that on the first turn of the back-and-forth game the $\forall$-player plays $\bar{b}=\bar{b}_1\bar{b}_2\bar{b}_3$ extending $x$, where $\bar{b}_1$ is contained in $\A$, $\bar{b}_2$ in $\omega$ and $\bar{b}_3$ in $\B$.
    Without loss of generality, $\bar{b}_2$ contains the first $m$ elements of $\omega$.
    Let $N=m+|\bar{b}|$.
    Find a block of size at least $N$ that is less than any block of any element of $\bar{b}_3$ (this is possible by assumption), and let $\bar{b}'$ be the first $m$ elements of this block.
    We claim that $\bar{b}_1\bar{b}_2\bar{b}_3\leq_0 \bar{b}_1\bar{b}'_2\bar{b}_3$.
    This is true because 
    \begin{enumerate}
        \item there are no functional relationships between any of $\bar{b}_1$, $\bar{b}_2$, $\bar{b}_3$, and  $\bar{b}'_2$,
        \item $\bar{b}_2$ and $\bar{b}'_2$ agree on the first $|\bar{b}|$ formulas (in each case every element has at least $|\bar{b}|$ successors and the same number of predecessors),
        \item $\bar{b}_1$ and $\bar{b}_3$ stay the same in both tuples.
    \end{enumerate}
    This shows that $x$ is 1-free and finishes the proof.

    The second half of the lemma follows by a symmetrical argument.
\end{proof}

Note, however, that this is the only barrier to being uniformly relatively categorical that there is.

\begin{lemma}\label{lem:allCasesSp}
    If $\L$ is an $sp$-homogeneous linear ordering such that $(\L,<)$ is uniformly $\Delta_3$ categorical and $(\L,<,s,p)$ is not in one of the cases described by Lemma \ref{lem:forbiddensp} then $(\L,<,s,p)$ is uniformly relatively categorical.
\end{lemma}

\begin{proof}
    If $(\L,<)$ is uniformly $\Delta_3$ categorical, then it avoids the situations listed in Theorem \ref{spud3cat}.
    Because $(\L,<,s,p)$ is not in any of the cases listed by Lemma \ref{lem:forbiddensp}, checking case by case, this means that every purely infinite interval $I$ must have $\L=\A'+\A+I+\B+\B'$ where $\A$ and $\B$ only use finitely many finite blocks. 
    Note that the formulas $\varphi_{n,<}$ and $\varphi_{n,>}$ from Proposition \ref{prop:Sigma3blockabove} are $\Sigma_1$ in $s$ and $p$.
    In particular, it is enough to show that, in the remaining cases, $\A+I+\B$ is uniformly, relatively, computably categorical. 
    As pointed out in the proof of Proposition \ref{prop:infiniteAnalysis} (3), each of the elements in a finite block or a block of type $\omega$ or $\omega^*$ has its orbits in $\A+I+\B$ defined by formulas of the form $P_n(x), \lnot P_n(x), S_m(x),\lnot S_m(x)$.
    Note that each of these is $\Sigma_1$ in $s$ and $p$ (in fact, quantifier-free), so the automorphism orbits in each of these block types are suitably defined.
    By Theorem \ref{spud3cat} (5) along with Lemma \ref{lem:forbiddensp}, $I$ either has a block of type $\omega$ or a shuffle sum of finitely many finite blocks to its right and a block of type $\omega^*$ or a shuffle sum of finitely many finite blocks to its left.
    In each case, $\zeta$ or $\zeta\cdot\eta$ is $\Sigma_1$ definable within the ordering, which yields the needed Scott set.
    It is always quantifier-free to define the first element of $\omega$ or the last of $\omega^*$, so it is $\Sigma_1(<,s,p)$ to say you are to the left or right (resp.) of these blocks.
    Saying that the elements of $\zeta$ or $\zeta\cdot\eta$ are not among a shuffle sum of finitely many finite blocks comes down to simply stating the existence of some amount of successors and predecessors, which is quantifier-free in $s$ and $p$.
    In total, this gives the desired Scott set for $(\L,<,s,p)$.
\end{proof}

Taken together, these lemmas yield the following result.

 \begin{theorem}
        A homogeneous $sp$-linear ordering $(\L,<,s,p)$ is uniformly relatively computably categorical if and only if it is relatively computably categorical if and only if: 
        \begin{enumerate}
            \item It has no intervals of the form $Sh(S)$ where $S$ includes an infinite block and finite blocks of arbitrary size or $\zeta$ along with another finite block.
            \item If $I$ is $\omega\cdot\eta$, $Sh(\omega,\omega^*)$, $\omega^*\cdot\eta$ or $\zeta\cdot\eta$ or a sum of at least two of those orderings, it has an interval to its left and right in $\L$ that only has finite blocks of a bounded size.
            \item Any $\zeta\cdot\eta$ does not have an interval to its right or left isomorphic to a shuffle sum including an infinite block.
        \end{enumerate}
    \end{theorem} 
   
\begin{proof}
    Note that any uniformly relatively computably categorical structure is also relatively computably categorical by definition.
    Furthermore, by Lemma \ref{lem:ConnectingSpTo<} and Theorem \ref{spd3}, all relatively computably categorical $(\L,<,s,p)$ must satisfy the three listed structural criteria.
    By Lemma \ref{lem:allCasesSp} and Theorem \ref{spud3cat}, any linear ordering with these criteria is uniformly relatively computably categorical as \textit{sp}-orderings.
\end{proof}

\bibliographystyle{amsplain}
\bibliography{dclo}

\end{document}